\documentclass[11pt,a4paper]{article}
\usepackage[margin=1.3in]{geometry}
\usepackage[T1]{fontenc}
\usepackage{tikz}
\usetikzlibrary{positioning}
\usepackage[utf8]{inputenc}
\usepackage{graphicx}
\usepackage{mathtools}
\usepackage{amssymb}
\usepackage{amsthm}
\usepackage{thmtools}
\usepackage{xcolor}
\usepackage{nameref}
\usepackage{hyperref}
\usepackage{amsmath}
\usepackage{amsfonts}
\usepackage{wrapfig}
\usepackage[autostyle]{csquotes}
\usepackage{layout}
\usepackage{tabularx}
\usepackage{authblk}
\usetikzlibrary{shapes.geometric, positioning}
\usepackage{caption}
\usepackage{mathrsfs}
\usepackage{subcaption}
\usetikzlibrary{positioning,chains,fit,shapes,calc}
\usepackage{tikz}
\theoremstyle{definition}
\newtheorem{defn}{Definition}[section]
\newtheorem{theorem}{Theorem}[section]
\newtheorem{corollary}{Corollary}[theorem]
\newtheorem{lemma}[theorem]{Lemma}

\theoremstyle{remark}
\newtheorem{remark}{Remark}[section]

\date{}
\begin{document}
		\title{\sc On Intersection Hypergraph of Finite Non-cyclic Abelian Groups}
	
	\author{{ \sc{ Ardra A N} ~  and ~  \sc Sachin Ballal}\footnote{Corresponding Author}}
	\affil{School of Mathematics and Statistics, University of Hyderabad,
		500046, India\\  23mmpp02@uohyd.ac.in, sachinballal@uohyd.ac.in}

	\date{}
	\maketitle
	
	\begin{abstract}
		In this paper, we study various structural properties like girth, chromatic number, etc. of  the intersection  hypergraph on the subgroups of finite non-cyclic abelian group and we  characterize for the star and hypertree structures. Furthermore, we  characterize  hypergraphs which are embeddable on plane, torus and projective plane.
	\end{abstract}
	\section{Introduction}
Associating graphs and hypergraphs with various algebraic structures is an active and interesting area of research in mathematics \cite{sb1,sb2,sb3,cameron2023hypergraphs,devicomplement,gh}. Since the collection of maximal cliques of a graph naturally forms a hypergraph, called the clique hypergraph, this motivates us to study clique hypergraphs of graphs arising algebraic structures. The clique hypergraph of a graph $\mathcal{G}$ is the hypergraph with vertex set $V(\mathcal{G})$ whose
edge set is the family of vertex sets of maximal cliques in graph $\mathcal{G}$.  Our studies in \cite{sb1,sb2,sb3} show that structural properties such as girth, chromatic number, hypertree structure, and planarity can differ significantly between a graph and its clique hypergraph. Hence, such a study is relevant.  For more details on graphs and hypergraphs, one may refer \cite{bretto,voloshinhypergraphs,white1985graphs}, etc.

	\begin{defn}\cite{sb1}\label{intdefn}
		Let $G$ be a group and $S$ be the set of all non-trivial proper subgroups of $G$. \textit{The intersection hypergraph of $G$}, denoted by $\tilde{\Gamma}_\mathcal{H}(G)$, is a hypergraph whose vertex set,\\ $V =\{H \in S \,\, | \,\, H \cap K = \{e\} \,\, \text{for some $ K \in S $} \}$ and $E \subseteq V$ is a hyperedge if and only if \begin{enumerate}
			\item For distinct $H,K \, \in E$, $H \cap K = \{e\}.$
			\item There does not exist  $E' \supset E$ which  satisfies (1).
		\end{enumerate}
		
	\end{defn}
		
	\begin{defn}\cite{devicomplement}
		The complement of the intersection graph of subgroups of a group G, denoted by  $\mathscr{I}^c(G)$, is the graph whose vertex set is the set of all nontrivial proper subgroups of
		G and its distinct vertices H and K are adjacent if and only if $H \cap K =\{e\}$.
	\end{defn} 
	Let  $\mathscr{I}^c(G)^{\star}$ denote the graph obtained from  $\mathscr{I}^c(G)$ by removing all the isolated vertices. Hence,  $\tilde{\Gamma}_\mathcal{H}(G)$ is the clique hypergraph of $\mathscr{I}^c(G)^{\star}$.

	In this paper, our aim is to study the intersection hypergraph on the subgroups of a finite non-cyclic abelian group.  In Section 2, we study various structural properties of intersection hypergraph of a  finite non-cyclic abelian group. In Section 3,  we characterize those hypergraphs which are embeddable on plane, torus and projective plane.\\
	Henceforth, we assume that $G$ is a finite non-cyclic abelian group.
	\section{Structural Properties of $\tilde{\Gamma}_\mathcal{H}(G)$}
In this section, we study various structural properties of $\tilde{\Gamma}_\mathcal{H}(G)$. Note that $\tilde{\Gamma}_\mathcal{H}(G)$ is non-empty because all the subgroups of $G$ with prime order have pairwise trivial intersection and constitute a hyperedge. From the Cauchy's theorem for finite abelian group, it is immediate that a subgroup $H$ of $G$ is in the vertex set of $\tilde{\Gamma}_\mathcal{H}(G)$ if and only if the intersection of $H$ with atleast one subgroup of prime order is trivial.
\begin{remark}\label{cyc}
	Every cyclic subgroup of $G$ belongs to the vertex set of $\tilde{\Gamma}_\mathcal{H}(G)$.
\end{remark}
\noindent We now analyze the hyperedge set of $\tilde{\Gamma}_\mathcal{H}(G)$. The following results determine some positive integers $n$ for which there exists a group $G$ such that $\mid E(\tilde{\Gamma}_\mathcal{H}(G))\mid =n$.
\begin{theorem}\label{singlehyperedgeset}
	$\mid E(\tilde{\Gamma}_\mathcal{H}(G))\mid = 1$ if and only if $G \cong \mathbb{Z}_{p_1} \times \mathbb{Z}_{p_1}$ for some prime $p_1$.
\end{theorem}
\begin{proof}
	If $G \cong \mathbb{Z}_{p_1} \times \mathbb{Z}_{p_1}$, where $p_1$ is a prime, then there exist exactly $p_1+1$ non-trivial proper subgroups of order $p_1$ and constitute a single hyperedge with $p_1+1$ vertices, i.e., $\mid E(\tilde{\Gamma}_\mathcal{H}(G))\mid = 1$.
	
	\noindent Conversely, assume that  $G \ncong \mathbb{Z}_{p_1} \times \mathbb{Z}_{p_1}$. Clearly, there exists a hyperedge $e_1$  consisting of  all the $p_1$ order subgroups of $G$. Now, consider the following cases:\\
	\textbf{Case 1.} Let $G \cong \mathbb{Z}_{{p_1}^{\alpha_1}} \times \mathbb{Z}_{{p_1}^{\alpha_2}} \times \ldots \times \mathbb{Z}_{{p_1}^{\alpha_k}}$, where  $\alpha_i's$ are non-negative integers.\\
	\textbf{Subcase 1.1.} If $G$ is an elementary abelian group, then observe that $\mathbb{Z}_{p_1} \times \mathbb{Z}_{p_1} \times \{0\} \times \cdots \times \{0\} \bigcap  <(0,0,1,0,\cdots,0)>=\{e\}$ and $\mathbb{Z}_{p_1} \times \mathbb{Z}_{p_1} \times \{0\} \times \cdots \times \{0\} \bigcap <(1,0,0,0,\cdots,0)> \neq \{e\}$. Hence, there exists a hyperedge different from $e_1$ that contains $\mathbb{Z}_{p_1} \times \mathbb{Z}_{p_1} \times \{0\} \times \cdots \times \{0\}$ and $<(0,0,1,0,\cdots,0)>$ and, consequently, $\mid E(\tilde{\Gamma}_\mathcal{H}(G))\mid \neq 1$. \\
	\textbf{Subcase 1.2.} If $G$ is not an elementary abelian group, then, by Remark \ref{cyc}, there exists a hyperedge $e_2$ different from $e_1$ such that it contains a cyclic subgroup of $G$ with order $p_1^2$. Thus, $\mid E(\tilde{\Gamma}_\mathcal{H}(G))\mid \neq 1$.\\
	\textbf{Case 2.} If $G \cong \mathbb{Z}_{{p_1}^{\alpha_1}} \times  \mathbb{Z}_{{p_2}^{\alpha_2}} \times \ldots \times \mathbb{Z}_{{p_k}^{\alpha_k}}$, where  $\alpha_i's$ are non-negative integers, $p_i$'s are primes, not all distinct and not all equal, then there exists a hyperedge $e_2$ different from $e_1$ such that $e_2$ contains a cyclic subgroup of order $p_1p_2$, where $p_1$ and $p_2$ are two distinct prime divisors of $| G |$, and consequently, $\mid E(\tilde{\Gamma}_\mathcal{H}(G))\mid \neq 1$. 
\end{proof}
\begin{theorem}\label{edgecard}
	If $n$ is  a positive integer which is either of the form $p+1$ or $p+3$, where $p$ is  prime, then there exists a group $G$ with $\mid E(\tilde{\Gamma}_\mathcal{H}(G))\mid =n$.
\end{theorem}
\begin{proof}
	If $n$ is of the form $p+1$, then consider the group $G \cong \mathbb{Z}_{p^2} \times \mathbb{Z}_p$. The hyperedges of $\tilde{\Gamma}_\mathcal{H}(G)$ are: $e_1 =\{<(0,1)>,<(p,1)>,<(p,2)>, \cdots, <(p,p-1)>, <(p,0)>\}, e_2=\{<(0,1)>,<(p,1)>, <(p,2)>, \cdots, <(p,p-1)>, <(1,0)>\}, e_3=\{<(0,1)>,<(p,1)>,<(p,2)>, \cdots, <(p,p-1)>, <(1,1)>\}, \cdots, e_{p+1}=\{<(0,1)>,<(p,1)>,<(p,2)>, \cdots, <(p,p-1)>, <(1,p-1)>\}$ and consequently, $\mid E(\tilde{\Gamma}_\mathcal{H}(G))\mid =p+1$.\\
	 If $n$ is of the form $p+3$, then consider the group $G \cong \mathbb{Z}_{pq} \times \mathbb{Z}_p$, where $q$ is a prime distinct from $p$. Let $H_1, H_2, \cdots, H_{p+1}$ be $p+1$ subgroups of $G$ with order $p$, $H$ be the subgroup of $G$ with order $p^2$, $H^{'}$ be the subgroup of $G$ with order $q$, and $H_1H^{'}, H_2H^{'}, \cdots, H_{p+1}H^{'}$ be $p+1$ subgroups of $G$ with order $pq$. Then, the hyperedges of $\tilde{\Gamma}_\mathcal{H}(G)$ are: $e_1 =\{H_1, H_2, H_3, \cdots, H_{p+1}, H^{'} \}, e_2=\{H_2,H_3,H_4, \cdots, H_1H^{'}\}, e_3=\{H_1,H_3,H_4, \cdots, H_2H^{'}\}, \cdots, e_{p+2}=\{H_1,H_2,H_3, \cdots, H_{p+1}H^{'}\}$ and $e_{p+3}=\{H,H^{'}\}$ and consequently, $\mid E(\tilde{\Gamma}_\mathcal{H}(G))\mid =p+3$.
\end{proof}
	The following theorem characterizes the star structures in $\tilde{\Gamma}_\mathcal{H}(G)$.
	\begin{theorem}\label{starhyp}
			For a finite non-cyclic abelian group $G$, $\tilde{\Gamma}_\mathcal{H}(G)$ is a star hypergraph if and only if  $G \cong \mathbb{Z}_{p_1^\alpha} \times \mathbb{Z}_{p_1}$, where $p_1$ is a prime and $\alpha$ is a positive integer.
	\end{theorem}
	\begin{proof}
			Assume that  $G \cong \mathbb{Z}_{p_1^\alpha} \times \mathbb{Z}_{p_1}$, where $p_1$ is a prime. Consider the following cases:\\
		\textbf{Case 1.}  If $\alpha=1$, by Theorem \ref{singlehyperedgeset}, $\tilde{\Gamma}_\mathcal{H}(G)$ is a star hypergraph.\\
		\textbf{Case 2.} If  $\alpha \geq 2$, then the vertex $<(0,1)>$ of $\tilde{\Gamma}_\mathcal{H}(G)$ is adjacent to all the vertices in $\tilde{\Gamma}_\mathcal{H}(G)$. Hence, $\tilde{\Gamma}_\mathcal{H}(G)$ is a star hypergraph.\\
		Conversely, assume that $G \ncong  \mathbb{Z}_{p_1^\alpha} \times \mathbb{Z}_{p_1}$. Since there exists a hyperedge containing only the subgroups of $G$ of prime order, no subgroup of $G$ with  non-prime order belongs to all the hyperedges of  $\tilde{\Gamma}_\mathcal{H}(G)$. Now, to prove that no subgroup of prime order belongs to all the hyperedges, we will show that for each prime $p$ order subgroup $H_p$ there exists a vertex $K$ in $\tilde{\Gamma}_\mathcal{H}(G)$ such that $H_p \cap K \neq \{e\}$. To prove this claim, consider the following cases:\\
			\textbf{Case 1.}
		Let $G \cong \mathbb{Z}_{{p_1}^{\alpha_1}} \times  \mathbb{Z}_{{p_1}^{\alpha_2}} \times \ldots \times \mathbb{Z}_{{p_1}^{\alpha_k}}$, where $p_1$ is a prime and $\alpha_i's$ are non-negative integers.\\
		\textbf{Subcase 1.1.} If $k=2$, then for a prime order subgroup $H_p=<(g_1,g_2)>$, choose the vertex \\$K=<(g_1p_1^{-1},g_2p_1^{-1})>$, which is of order $p_1^2$ and so,  $H_p \cap K \neq \{e\}$.\\
		\textbf{Subcase 1.2.} If $k \geq 3$, then $G$ contains more than $p_1+1$ subgroups of order $p_1$. Observe that every non cyclic subgroup of $G$ of order $p_1^2$ has exactly $p_1+1$ subgroups of order $p_1$. Hence, for each subgroup $H_{p_1}$ of order $p_1$, there exists a subgroup $K \in V(\tilde{\Gamma}_\mathcal{H}(G))$ of order $p_1^2$ such that $H_{p_1} \cap K \neq \{e\}$.\\
		\textbf{Case 2.} Let $G \cong \mathbb{Z}_{{p_1}^{\alpha_1}} \times  \mathbb{Z}_{{p_2}^{\alpha_2}} \times \ldots \times \mathbb{Z}_{{p_k}^{\alpha_k}}$, where  $\alpha_i's$ are non-negative integers, $p_i$'s are primes, not all distinct and not all equal. For a prime $p_1$ order subgroup $H_{p_1}$, choose $K = H_{p_1}H_{p_2}$, where $p_2$ is a prime different from $p_1$.\\
		Hence, from the cases above, we can conclude that $\tilde{\Gamma}_\mathcal{H}(G)$ is not a star hypergraph.
	\end{proof}
	\begin{corollary}
		There does not exist a group $G$ with $\mid E(\tilde{\Gamma}_\mathcal{H}(G))\mid = 2$. 
	\end{corollary}
	\begin{proof}
		Suppose that there exists a group $G$ with $\mid E(\tilde{\Gamma}_\mathcal{H}(G))\mid = 2$. Then, $\tilde{\Gamma}_\mathcal{H}(G)$ is a star. Hence, by Theorem \ref{starhyp}, $G \cong \mathbb{Z}_{p_1^\alpha} \times \mathbb{Z}_{p_1}$, where $p_1$ is a prime and $\alpha$ is a positive integer. From Theorem \ref{singlehyperedgeset} and by Theorem \ref{edgecard}, $\alpha=1$ and $\alpha=2$ cases do not arise.\\
		Now, if $\alpha \geq 3$, then there exist vertices  $<(p_1,0)>$ and $<(1,0)>$ of $\tilde{\Gamma}_\mathcal{H}(G)$ such that $<(p_1,0)> \cap <(1,0)> \neq \{e\}$. Hence, there exist hyperedges $e_1$ and $e_2$ containing $<(p_1,0)>$ and $<(1,0)>$ respectively and one hyperedge $e_3$ containing all the subgroups of order $p_1$. Consequently,  $\mid E(\tilde{\Gamma}_\mathcal{H}(G))\mid \geq 3$, which is a contradiction.	Therefore, there does not exist a group $G$ with $\mid E(\tilde{\Gamma}_\mathcal{H}(G))\mid = 2$. 
	\end{proof}
	To characterize $\tilde{\Gamma}_\mathcal{H}(G)$ is a hypertree, we essentially need the following.
	\begin{defn}\cite{voloshinhypergraphs}
		A \textit{host graph} $\mathcal{G}$ for a hypergraph is a connected graph on the same vertex set such that
		every hyperedge induces a connected subgraph of $\mathcal{G}$. A hypergraph $\mathcal{H} = (X,\mathcal{D})$ is called a \textit{hypertree} if there exists a host tree
		$T = (X,E)$ such that each edge $D \in \mathcal{D}$  induces a subtree in $T$. 
	\end{defn}
	\noindent Note that a simple graph $\mathcal{G}$ is a tree if and only if there is a unique path connecting any two vertices of $\mathcal{G}$.
	\begin{theorem}
	 $\tilde{\Gamma}_\mathcal{H}(G)$ is a hypertree if and only if $G \cong \mathbb{Z}_{p_1^\alpha} \times \mathbb{Z}_{p_1}$, where $p_1$ is a prime and $\alpha \geq 2$.
	\end{theorem}
	\begin{proof}
		If  $G \cong \mathbb{Z}_{p_1^\alpha} \times \mathbb{Z}_{p_1}$, then by Theorem \ref{starhyp}, $\tilde{\Gamma}_\mathcal{H}(G)$ is a star hypergraph and consequently,  $\tilde{\Gamma}_\mathcal{H}(G)$ is a hypertree.
		
		\noindent Conversely, assume that $G \ncong \mathbb{Z}_{p_1^\alpha} \times \mathbb{Z}_{p_1}$. Suppose that $G_\mathcal{H}$ is a host graph for $\tilde{\Gamma}_\mathcal{H}(G)$. We claim that for a pair of vertices of $G_\mathcal{H}$ there exist more than one path connecting them in $G_\mathcal{H}$. Consider the following cases:\\
		\textbf{Case 1.} Let $G \cong \mathbb{Z}_{{p_1}^{\alpha_1}} \times  \mathbb{Z}_{{p_1}^{\alpha_2}} \times \ldots \times \mathbb{Z}_{{p_1}^{\alpha_k}}$, where $p_1$ is a prime and $\alpha_i's$ are non-negative integers. Since all the subgroups of $G$ with order $p_1$ constitute a hyperedge,  all these subgroups must form a connected subgraph $T$ in  $G_\mathcal{H}$.  Let $H_1$ and $H_2$ be two arbitrary subgroups of $G$ with order $p_1$ such that $dist(H_1,H_2)=2$ in $T$.  Hence, there exists a path $P$ from $H_1$ to $H_2$ in $T$ of $G_{\mathcal{H}}$. Suppose $H_3$ is a vertex in $P$ between $H_1$ and $H_2$. Then, there exists a subgroup $K \supseteq H_3$ such that $K$ is adjacent to all subgroups of $G$ with order $p_1$, except $H_3$. Hence, there exists an another path from $H_1$ to $H_2$  containing $K$, but not containing $H_3$ in $G_{\mathcal{H}}$.\\ 
		\textbf{Case 2.} Let $G \cong \mathbb{Z}_{{p_1}^{\alpha_1}} \times  \mathbb{Z}_{{p_2}^{\alpha_2}} \times \ldots \times \mathbb{Z}_{{p_k}^{\alpha_k}}$, where  $\alpha_i's$ are non-negative integers and $p_i$'s are primes, not all distinct and not all equal. Without loss of generality, assume that $p_1$ is repeating atleast twice.  Since all the subgroups of $G$ with prime order constitute a hyperedge,  all these subgroups must form a connected subgraph $T$ in  $G_\mathcal{H}$.  Let $H_1$ and $H_2$ be two arbitrary subgroups of $G$ with order $p_1$ such that $dist(H_1,H_2)=2$ in $T$.  Suppose $H_3$ is a vertex in the path between $H_1$ and $H_2$.\\
		\textbf{Case 2.1.} If $H_3$ is a subgroup of order $p_1$, then there exists a vertex $H_3H_{p_2}$, where $p_2$ is a prime different from $p_1$, which is adjacent to $H_1$ and $H_2$, but not to $H_3$. Hence, there exists an another path from $H_1$ to $H_2$ containing $H_3H_{p_2}$, but not $H_3$ in $G_{\mathcal{H}}$. \\
		\textbf{Case 2.2.} If $H_3$ is not of order $p_1$, then there exists a vertex $K$ of order $\dfrac{\mid G \mid}{p_1^{\alpha}}$, where $\alpha$ is the highest power of $p_1$ such that $p_1^{\alpha}$ divides $\mid G \mid$  which is adjacent to $H_1$ and $H_2$, but not to $H_3$. Hence, there exists an another path from $H_1$ to $H_2$ containing $K$, but not containing $H_3$ in $G_{\mathcal{H}}$.\\
		Thus, $G_\mathcal{H}$ is not a tree. Since $G_\mathcal{H}$ is an arbitrary  host graph for $\tilde{\Gamma}_\mathcal{H}(G)$, we can conclude that $\tilde{\Gamma}_\mathcal{H}(G)$ is not a hypertree.
		\end{proof}
	 \noindent In the next theorem, we investigate the girth of $gr(\tilde{\Gamma}_\mathcal{H}(G))$.
\begin{theorem}
	The girth $gr(\tilde{\Gamma}_\mathcal{H}(G)) $ is either 2 or $\infty$.
\end{theorem}
\begin{proof}
	Consider the following cases:\\
	\textbf{Case 1.} If $G \cong \mathbb{Z}_{p_1} \times \mathbb{Z}_{p_1}$, where $p_1$ is a prime, then $\tilde{\Gamma}_\mathcal{H}(G)$ is a hypergraph consisting of a single hyperedge with $p_1+1$ vertices. Hence, $gr(\tilde{\Gamma}_\mathcal{H}(G)) = \infty$.\\
	\textbf{Case 2.} If $G \cong \mathbb{Z}_{p_1^2} \times \mathbb{Z}_{p_1}$, where $p_1$ is a prime, then  $<(0,1)> e_1 <(p,1)> e_2 <(0,1)>$, where $e_1$ and $e_2$  are hyperdges of $\tilde{\Gamma}_\mathcal{H}(G)$ with $e_1 \supseteq \{<(0,1)>, <(p,1)>, <(p,0)>\}$ and $e_2 \supseteq \{<(0,1)>, <(p,1)>,<(1,0)>\}$, is a cycle of length 2 in $\tilde{\Gamma}_\mathcal{H}(G)$ and, consequently $gr(\tilde{\Gamma}_\mathcal{H}(G)) =2$.\\
	\textbf{Case 3.} If $G \cong \mathbb{Z}_{p_1} \times \mathbb{Z}_{p_1} \times \mathbb{Z}_{p_1}$, where $p_1$ is a prime, then  $<(1,0,0)>e_1<(1,1,0)>e_2<(1,0,0)>$, where $e_1$ and $e_2$ are hyperedges with  $e_1 \supseteq \{<(1,0,0)>,<(1,1,0)>, <(1,0,1)>\} $ and $e_2 \supseteq \{<(1,0,0)>,<(1,1,0)>, <(1,0,2)>\}$, is a cycle in $\tilde{\Gamma}_\mathcal{H}(G)$ of length 2 and, consequently, $gr(\tilde{\Gamma}_\mathcal{H}(G)) =2$.\\
	\textbf{Case 4.} If $G \cong \mathbb{Z}_{p_1} \times \mathbb{Z}_{p_1} \times \mathbb{Z}_{p_2}$, where $p_1$ and $p_2$ are distinct primes,  then $<(1,0,0)>e_1 <(0,1,0)> e_2 <(1,0,0)>$, where $e_1$ and $e_2$ are hyperedges with $e_1 \supseteq \{<(1,0,0)>,<(0,1,0)>, <(0,0,1)>\}$ and $e_2 \supseteq \{<(1,0,0)>,<(0,1,0)>, <(0,0,1),(1,1,0)>\}$, is a cycle in $\tilde{\Gamma}_\mathcal{H}(G)$ of length 2 and, consequently, $gr(\tilde{\Gamma}_\mathcal{H}(G)) =2$.\\
	\textbf{Case 5.} If $G \cong \mathbb{Z}_{{p_1}^{\alpha_1}} \times  \mathbb{Z}_{{p_2}^{\alpha_2}} \times \ldots \times \mathbb{Z}_{{p_k}^{\alpha_k}} \ncong \mathbb{Z}_{p_1} \times \mathbb{Z}_{p_1}$, where $p_i$'s are primes with atleast two $p_i's$ are equal and $\alpha_i's$ are non-negative integers, then there is an injective hypergraph homomorphism from  $\tilde{\Gamma}_\mathcal{H}(G_1)$ to $\tilde{\Gamma}_\mathcal{H}(G)$, where $G_1$ is one of the following group:  (i) $\mathbb{Z}_{p_1^2} \times \mathbb{Z}_{p_1}$, (ii) $\mathbb{Z}_{p_1} \times \mathbb{Z}_{p_1} \times \mathbb{Z}_{p_1}$ or (iii) $\mathbb{Z}_{p_1} \times \mathbb{Z}_{p_1} \times \mathbb{Z}_{p_2}$. Since there exists an injective hypergraph homomorphism from $\tilde{\Gamma}_\mathcal{H}(G_1)$ to $\tilde{\Gamma}_\mathcal{H}(G)$, $gr(\tilde{\Gamma}_\mathcal{H}(G_1)) \leq gr(\tilde{\Gamma}_\mathcal{H}(G))$ and hence, in this case,  $gr(\tilde{\Gamma}_\mathcal{H}(G))=2$.
\end{proof}
\noindent The chromatic number of $\mathscr{I}^c(G)^{\star}$ is greater than or equal to 3 but in case of its clique hypergraph, $\tilde{\Gamma}_\mathcal{H}(G)$, the chomatic number is 2 always.
\begin{theorem}
	The chromatic number, $\chi(\tilde{\Gamma}_\mathcal{H}(G))$,  is 2.
\end{theorem}
\begin{proof}
	Consider the following cases:\\
	\textbf{Case 1.} Let $G \cong \mathbb{Z}_{p_1^{\alpha_1}} \times \mathbb{Z}_{p_2^{\alpha_2}} \times \mathbb{Z}_{p_3^{\alpha_3}} \times \cdots \times  \mathbb{Z}_{p_k^{\alpha_{k}}}$, where $p_i$'s are primes, not all distinct and not all equal and $\alpha_i's$ are non-negative integers. Without loss of generality, assume that $p_1$ is repeating atleast twice. Divide the vertex set of $\tilde{\Gamma}_\mathcal{H}(G)$ into two sets $A$ and $B$ as follows:\\
	$$A =\{H \in V(\tilde{\Gamma}_\mathcal{H}(G)) \, \, : \,\,  \mid H \mid =p_1^\alpha \,\,\, \text{for some positive integer $\alpha$} \}$$
	and $$B =\{H \in V(\tilde{\Gamma}_\mathcal{H}(G)) \, \, : \,\,  \mid H \mid \neq p_1^\alpha \,\,\, \text{for some positive integer $\alpha$} \}.$$
	In order to show that $\chi(\tilde{\Gamma}_\mathcal{H}(G))=2$, we need to show that each hyperedge of $\tilde{\Gamma}_\mathcal{H}(G)$ contains atleast one vertex from the set $A$ and atleast one from the set $B$. If not, i.e., if $e_1$ contains the vertices from $A$ only, then we can find a subgroup $H_{p_2} \in B$ of $G$ with order $p_2$ with $H_{p_2} \cap H =\{e\}$ for all $H \in A$. Thus, $H_{p_2}$ must belong to $e_1$, otherwise contradicts the maximality condition of $e_1$. Similarly, $e_2$ cannot contain vertices from the set $B$ only. Hence, $\chi(\tilde{\Gamma}_\mathcal{H}(G))=2$.\\
	\textbf{Case 2.} Let $G \cong \mathbb{Z}_{p_1^{\alpha_1}} \times \mathbb{Z}_{p_1^{\alpha_2}} \times \mathbb{Z}_{p_1^{\alpha_3}} \times \cdots \times \mathbb{Z}_{p_1^{\alpha_k}}$, where $p_1$ is a prime, $\alpha_i$'s are non-negative integers and $k \geq 2$. Divide the subgroups of $G$ with order $p_1$ into two sets $A^{'}$ and $B^{'}$ such that either $\mid A^{'} \mid = \mid B^{'} \mid$ or  $\mid A' \mid = \mid B' \mid + \, \, 1 $. Divide $V(\tilde{\Gamma}_\mathcal{H}(G))$ into two sets $A$ and $B$ as follows:\\
	$$A=\{H \in V(\tilde{\Gamma}_\mathcal{H}(G)): \not\exists K \in B^{'} \, \text{such that} \, K \leq H\} \quad \quad \text{and} \quad \quad B =  V(\tilde{\Gamma}_\mathcal{H}(G)) - A$$ 
	Note that $A' \subseteq A$ and $B' \subseteq B$ and therefore $A$ and $B$ are non-empty. Then, on similar lines as in \textbf{Case 1}, we can conclude that $\chi(\tilde{\Gamma}_\mathcal{H}(G))=2$.
	\end{proof}
	\section{Embedding of $\tilde{\Gamma}_\mathcal{H}(G)$}
	An \textit{embedding of a graph} on a surface is a continuous one-to-one map from a topological representation of the graph into the surface, such that edges intersect only at their common endpoints.  We denote by $S_n$ the surface obtained from the sphere $S_0$  by adding $n$ handles. The number $n$ is called the \textit{genus of the surface} $S_n, n\geq 0$.  The \textit{orientable genus} of a graph $G$, denoted by $g(G)$, is the minimum genus of
	a surface in which $G$ can be embedded.  A \textit{cross-cap} is a topological object formed by identifying opposite points on the boundary of a circle (or a disk) and is equivalent to gluing a Möbius strip into a hole in a surface. A surface obtained by adding $k$ crosscaps to $S_0$ 
	is known as the \textit{non-orientable surface} and we denote it by $N_k$. The number $k$ is called the \textit{crosscap} of $N_k$. The \textit{non-orientable genus} of
	a graph $G$, denoted by $\tilde{g}(G)$, is the smallest integer $k$ such that $G$ can be embedded on $N_k$.  A graph is said to be \textit{planar} if it can be drawn on the plane in such a way that no edges intersect, except at a common end vertex.  A graph is said to be \textit{toroidal} if it  can be embedded on a torus and is called \textit{projective} if it can be embedded on a projective plane.   Further, note that if $H$ is a subgraph of a graph $G$, then $g(H) \leq g(G)$ and $\tilde{g}(H) \leq \tilde{g}(G)$. A hypergraph is \textit{toroidal} if its incidence graph is toroidal and is  \textit{projective} if its incidence graph is projective. 
	
	Embeddings of hypergraphs on surfaces reveal a relationship between combinatorics and topology. In this section, we characterize the hypergraphs that can be embedded in the plane, the torus, and the projective plane.
	
	\begin{lemma}\label{injgenus}
			If there exists an injective hypergraph homomorphism from $\tilde{\Gamma}_\mathcal{H}(G_1)$ to  $\tilde{\Gamma}_\mathcal{H}(G_2)$, then $g(\tilde{\Gamma}_\mathcal{H}(G_1)) \leq g(\tilde{\Gamma}_\mathcal{H}(G_2))$ and 	$\tilde{g}(\tilde{\Gamma}_\mathcal{H}(G_1)) \leq \tilde{g}(\tilde{\Gamma}_\mathcal{H}(G_2))$.
	\end{lemma}
	\begin{proof}
	An injective homomorphism from $\tilde{\Gamma}_\mathcal{H}(G_1)$ to  $\tilde{\Gamma}_\mathcal{H}(G_2)$ embeds $\tilde{\Gamma}_\mathcal{H}(G_1)$ as a subhypergraph of  $\tilde{\Gamma}_\mathcal{H}(G_2)$.  Thus, $\mathcal{I}(\tilde{\Gamma}_\mathcal{H}(G_1))$ is embedded as a subgraph of $\mathcal{I}(\tilde{\Gamma}_\mathcal{H}(G_2))$. Hence, $g(\tilde{\Gamma}_\mathcal{H}(G_1)) \leq g(\tilde{\Gamma}_\mathcal{H}(G_2))$ and 	$\tilde{g}(\tilde{\Gamma}_\mathcal{H}(G_1)) \leq \tilde{g}(\tilde{\Gamma}_\mathcal{H}(G_2))$.
	\end{proof}
		\noindent First, we analyze $\tilde{\Gamma}_\mathcal{H}(G)$ which are planar. To examine the planarity of $\tilde{\Gamma}_\mathcal{H}(G)$, we rely on the following results. 
	\begin{theorem}\label{kurplanar}\cite{walsh1975hypermaps}
		A graph  is planar if and only if it contains no subdivision of $K_5$ or $K_{3,3}$.
	\end{theorem}
	\begin{theorem}\label{planarhyper}\cite{walsh1975hypermaps} 
		A hypergraph is planar if and only if its incidence graph is planar.
	\end{theorem}
	\begin{theorem}\cite{sb3}\label{planarcyclic}
		Let $n=p_1^{\alpha_1}p_2^{\alpha_2} \ldots p_k^{\alpha_k}$, where $p_1, p_2, \ldots, p_k$ are distinct primes and $\alpha_1, \alpha_2, \cdots, \alpha_k$ are non-negative integers. Then,  $\tilde{\Gamma}_\mathcal{H}(\mathbb{Z}_n)$ is planar if and only if $n$ is equal to  $p_1^{\alpha_1}p_2$ or $p_1^{\alpha_1}p_2^2$ or  $p_1p_2p_3$ or $p_1^2p_2p_3$.
	\end{theorem}
	\begin{theorem}\label{planar}
		$\tilde{\Gamma}_\mathcal{H}(G)$ is planar if and only if $G$ is one of the following:
		\begin{enumerate}
		\item $\mathbb{Z}_{p_1} \times \mathbb{Z}_{p_1}$, where $p_1$ is any prime, 
		\item $ \mathbb{Z}_{2^\alpha} \times \mathbb{Z}_2$, where $\alpha \geq 2$ is a positive integer,
		\item $\mathbb{Z}_2 \times \mathbb{Z}_2 \times \mathbb{Z}_{p_2^\beta}$, where $\beta =1,2$ and $p_2$ is an odd prime.
		\end{enumerate}
	\end{theorem}
\begin{proof}
	First, consider the following cases: \\
	\textbf{Case 1.} If $G \cong \mathbb{Z}_{p_1} \times \mathbb{Z}_{p_1}$, for any prime $p_1$, then from Theorem \ref{singlehyperedgeset},  $\tilde{\Gamma}_\mathcal{H}(G)$ is planar.\\
	\textbf{Case 2.} If $G \cong \mathbb{Z}_{2^\alpha} \times \mathbb{Z}_2$, where $\alpha \geq 2$ is a positive integer, then Figure \ref{planar1}(a) depicts an embedding of $\mathcal{I}(\tilde{\Gamma}_\mathcal{H}(G))$ on a plane and hence, $\tilde{\Gamma}_\mathcal{H}(G)$ is planar.\\
		\textbf{Case 3.} If $G \cong \mathbb{Z}_2 \times \mathbb{Z}_2 \times \mathbb{Z}_{p_2^\beta}$, where $p_2$ is an odd prime and $\beta =1,2$ ,  then $\tilde{\Gamma}_{\mathcal{H}}(G)$ is a subhypergraph of $\tilde{\Gamma}_\mathcal{H}(\mathbb{Z}_2 \times \mathbb{Z}_2 \times \mathbb{Z}_{p_2^2})$. Consider the following hyperedges of $\tilde{\Gamma}_\mathcal{H}(\mathbb{Z}_2 \times \mathbb{Z}_2 \times \mathbb{Z}_{p_2^2})$: $e_1=\{<(1,1,0)>,<(1,0,0)>,<(0,1,0)>,<(0,0,p_2)>\}, e_2=\{<(1,1,0)>,<(1,0,0)>,<(0,1,0)>,<(0,0,1)>\}, e_3=\{<(1,1,0)>,<(1,0,0)>,<(0,1,0),(0,0,p_2)>\}, e_4=\{<(1,1,0)>,<(0,1,0)>,<(1,0,0),(0,0,p_2)>\}, e_5=\{<(1,0,0)>,<(0,1,0)>,<(1,1,0),(0,0,p_2)>\}$ and $e_6 =\{\mathbb{Z}_2 \times \mathbb{Z}_2 \times \{0\}, <(0,0,p_2)>\}$. Figure \ref{planar1}(b) depicts an embedding of $\mathcal{I}(\tilde{\Gamma}_\mathcal{H}(\mathbb{Z}_2 \times \mathbb{Z}_2 \times \mathbb{Z}_{p_2^2}))$ on a plane and hence, $\tilde{\Gamma}_{\mathcal{H}}(G)$ is planar.
			\begin{figure}[htbp!]
			\centering	
			\subfloat[An embedding of $\mathcal{I}(\tilde{\Gamma}_\mathcal{H}(\mathbb{Z}_{2^\alpha} \times \mathbb{Z}_2))$ on a plane]{\begin{tikzpicture}[scale=.8, every node/.style={circle, fill=black, minimum size=4pt, inner sep=0pt}]
					
					
					\node [scale=.4,label=right:{}] (u4) at (0,2.5) {};
					\node [scale=.4,label=right:{}] (u5) at (0,3) {};
					\node [label=right:{\scriptsize{$<(0,1)>$}}] (a1) at (2,-1.5) {};
					\node [label=left:{\scriptsize{$<(2,1)>$}}] (a2) at (-2,-1.5) {};
					\node [label=below:{}] (v1) at (0,-.7) {};
					\node [label=below:{}] (u1) at (0,-.2) {};
					\node [label=below:{}] (v2) at (0,1.5) {};
					\node [label=below:{}] (u2) at (0,2) {};
					\node [label=below:{}] (v3) at (0,.4) {};
					\node [label=below:{}] (u3) at (0,.9) {};

					
					\draw (a1) -- (v1);
					\draw (a2) -- (v1);
					\draw (u1) -- (v1);
					\draw (a1) -- (v2);
					\draw (a2) -- (v2);
					\draw (u2) -- (v2);
					\draw (a1) -- (v3);
					\draw (a2) -- (v3);
					\draw (u3) -- (v3);

			\end{tikzpicture}}
			\hspace{3cm}
			\subfloat[$\mathcal{I}(\tilde{\Gamma}_\mathcal{H}(\mathbb{Z}_2 \times \mathbb{Z}_2 \times \mathbb{Z}_{p_2^2})$ on a plane]{\begin{tikzpicture}[scale=.6, every node/.style={circle, fill=black, minimum size=4pt, inner sep=0pt}]
					

					\node [label=left:{\tiny{<(1,0,0)>}}] (u1) at (-4,0) {};
					\node [label=left:{\tiny{<(1,1,0)>}}] (u2) at (0,0) {};
					\node [label=left:{\tiny{<(0,1,0)>}}] (u3) at (4,0) {};
					\node [label=right:{\tiny{<(0,0,q)>}}] (u4) at (4,2.5) {};
					\node [label=right:{\tiny{<(0,0,1)>}}] (u5) at (1,-2) {};
					\node [label=left:{\tiny{<(0,1,0),(0,0,q)>}}] (u6) at (-1.5,2) {};
					\node [label=right:{\tiny{<(1,0,0),(0,0,q)>}}] (u7) at (1,2) {};
					\node [label=right:{\tiny{<(1,0,0),(0,0,q)>}}] (u8) at (1,-4) {};
					\node [label=right:{$\mathbb{Z}_2 \times \mathbb{Z}_2 \times \{0\}$}] (u9) at (5,0) {};
					
					\node [label=above:{$e_1$}] (v1) at (0,4) {};
					\node [label=left:{$e_2$}] (v2) at (0,-2) {};
					\node [label=left:{$e_3$}] (v3) at (-2,1) {};
					\node [label=left:{$e_4$}] (v4) at (2,1) {};
					\node [label=left:{$e_5$}] (v5) at (0,-4) {};
					\node [label=right:{$e_6$}] (v6) at (5,1.5) {};


					\draw (u1) --(v1);
					\draw (u2) --(v1);
					\draw (u3) --(v1);
					\draw (u4) --(v1);
					\draw (u1) --(v2);
					\draw (u2) --(v2);
					\draw (u3) --(v2);
					\draw (u5) --(v2);
					\draw (u1) --(v3);
					\draw (u2) --(v3);
					\draw (u6) --(v3);
					\draw (u2) --(v4);
					\draw (u3) --(v4);
					\draw (u7) --(v4);
					\draw (u1) --(v5);
					\draw (u3) --(v5);
					\draw (u8) --(v5);
					\draw (u4) --(v6);
					\draw (u9) --(v6);
			\end{tikzpicture}}
		\hspace{3cm}
		\subfloat[A subgraph of $\mathcal{I}(\tilde{\Gamma}_{\mathcal{H}}(\mathbb{Z}_{2^2} \times \mathbb{Z}_{2^2}))$]{\begin{tikzpicture}[scale=.6, every node/.style={circle, fill=black, minimum size=4pt, inner sep=0pt}]
				

				\node [label=left:{<(0,2)>}] (u1) at (-1,2) {};
				\node [label=left:{<(2,0)>}] (u2) at (-1,0) {};
				\node [label=left:{<(2,2)>}] (u3) at (-1,-2) {};
				\node [label=right:{<(1,-1)>}] (u4) at (4,2) {};
				\node [label=right:{<(3,-1)>}] (u5) at (3,-2) {};
				\node [label=below:{<(1,0)>}] (u6) at (4,-4) {};
				\node [label=above:{\tiny{<(1,2)>}}] (u7) at (0.8,-3) {};

				\node [label=above:{$e_1$}] (v1) at (2,2) {};
				\node [label=right:{$e_2$}] (v2) at (2,0) {};
				\node [label=left:{$e_3$}] (v3) at (2,-2) {};
				\node [label=left:{$e_4$}] (v4) at (6,-1) {};
				\node [label=left:{$e_5$}] (v5) at (-.5,-4) {};
				\node [label=left:{$e_6$}] (v6) at (0,-3) {};
				\node [label=right:{$e_7$}] (v7) at (2,-3) {};


				\draw (u1) --(v1);
				\draw (u2) --(v1);
				\draw (u3) --(v1);
				\draw (u1) --(v2);
				\draw (u2) --(v2);
				\draw (u4) --(v2);
				\draw (u1) --(v3);
				\draw (u2) --(v3);
				\draw (u5) --(v3);
				\draw (u4) --(v4);
				\draw (u6) --(v4);
				\draw (u3) --(v5);
				\draw (u6) --(v5);
				\draw (u3) --(v6);
				\draw (u7) --(v6);
				\draw (u5) --(v7);
				\draw (u7) --(v7);
				
		\end{tikzpicture}}
			\caption{}
			\label{planar1}
		\end{figure}\\
		Conversely, assume that $\tilde{\Gamma}_{\mathcal{H}}(G)$ is planar. Observe that  $\tilde{\Gamma}_{\mathcal{H}}(G)$ is planar if and only if  $\tilde{\Gamma}_{\mathcal{H}}(K)$ is  planar for all $K \leq G$. Let $K \cong \mathbb{Z}_{p_1^{\alpha_1}} \times \mathbb{Z}_{p_2^{\alpha_2}} \times \mathbb{Z}_{p_3^{\alpha_3}} \times \cdots \times \mathbb{Z}_{p_k^{\alpha_k}}\leq G$, where $p_i'$s are prime numbers and $\alpha_i'$s are non-negative integers. We consider possibilities of $K$ as follows:\\
		\textbf{Case 1.} All $p_i's$ are distinct \\
		\textbf{Case 1.1.} If $K \cong \mathbb{Z}_{p_1^{\alpha_1}} \times \mathbb{Z}_{p_2^{\alpha_2}}$, where $\alpha_1, \alpha_2 \geq 3$, then, by Theorem \ref{planarcyclic},  $K$ is non-planar.\\
		\textbf{Case 1.2.} If $K \cong \mathbb{Z}_{p_1^3} \times \mathbb{Z}_{p_2} \times \mathbb{Z}_{p_3}$ or $K \cong \mathbb{Z}_{p_1^2} \times \mathbb{Z}_{p_2^2} \times \mathbb{Z}_{p_3}$, then, by Theorem 4.4 in \cite{sb3},  $K$ is non-planar.\\
		\textbf{Case 1.3.} If $k \geq 4$, then, by Theorem \ref{planarcyclic},  $K$ is non-planar.\\
		\textbf{Case 2.} Two $p_i'$s are same\\
		\textbf{Case 2.1.} If $K \cong \mathbb{Z}_{p_1^{\alpha_1}} \times \mathbb{Z}_{p_1^{\alpha_2}}$, then consider the following cases:\\
		\textbf{Case 2.1.1.} If $p_1=2$ and $\alpha_1, \alpha_2 \geq 2$, then $ \mathbb{Z}_{2^2} \times \mathbb{Z}_{2^2}$ is a subgroup of $K$. Figure \ref{planar1}(c) depicts that $\mathcal{I}(\tilde{\Gamma}_\mathcal{H}(\mathbb{Z}_{2^2} \times \mathbb{Z}_{2^2}))$ contains a $K_{3,3}$ subdivision and hence, $\tilde{\Gamma}_\mathcal{H}(K)$ is not planar.\\
		\textbf{Case 2.1.2.} If  $p_1 \geq 3$ and $\alpha_1, \alpha_2$ are  positive integers with $\alpha_1 \geq 2$, then $\mathbb{Z}_{p_1^2} \times \mathbb{Z}_{p_1}$ is subgroup of $K$. Consider the hyperedges $e_1,e_2$ and $e_3$ of $\tilde{\Gamma}_\mathcal{H}(\mathbb{Z}_{p_1^2} \times \mathbb{Z}_{p_1})$ such that $e_1 \supseteq \{<(0,1)>,<(p_1,1)>, <(p_1,2)>, <(p_1,0)>\}$, $e_2 \supseteq \{<(0,1)>,<(p_1,1)>, <(p_1,2)>, <(1,0)>\}$ and $e_3 \supseteq \{<(0,1)>,<(p_1,1)>, <(p_1,2)>, <(1,1)>\}$.  Thus, $\mathcal{I}(\tilde{\Gamma}_\mathcal{H}(\mathbb{Z}_{p_1^2} \times \mathbb{Z}_{p_1}))$ contains a $K_{3,3}$ and hence, by Theorem \ref{kurplanar} and Theorem \ref{planarhyper}, $\tilde{\Gamma}_\mathcal{H}(K)$ is not planar.\\
		\textbf{Case 2.2.} If $K \cong \mathbb{Z}_{p_1^\alpha} \times \mathbb{Z}_{p_1} \times \mathbb{Z}_{p_2^\beta}$, then consider the following cases:\\
		\textbf{Case 2.2.1.} If $p_1=2, \alpha=1$ and $\beta \geq 3$, then $K$ contains a subgroup isomorphic to  $\mathbb{Z}_2 \times \mathbb{Z}_2 \times \mathbb{Z}_{p_2^3}$. Consider the hyperedges $e_1,e_2$ and $e_3$ of $\tilde{\Gamma}_\mathcal{H}(\mathbb{Z}_2 \times \mathbb{Z}_2 \times \mathbb{Z}_{p_2^3})$ such that $e_1 \supseteq \{<(1,1,0)>,<(1,0,0)>, <(0,1,0)>, <(0,0,1)>\}$, $e_2 \supseteq \{<(1,1,0)>,<(1,0,0)>, <(0,1,0)>, <(0,0,p_2)>\}$ and $e_3 \supseteq \{<(1,1,0)>,<(1,0,0)>, <(0,1,0)>, <(0,0,p_2^2)>\}$. Thus, $\mathcal{I}(\tilde{\Gamma}_\mathcal{H}(\mathbb{Z}_2 \times \mathbb{Z}_2 \times \mathbb{Z}_{p_2^3}))$ contains a $K_{3,3}$ and hence, by Theorem \ref{kurplanar}, Theorem \ref{planarhyper} and Lemma \ref{injgenus}, $\tilde{\Gamma}_\mathcal{H}(K)$ is not planar.\\
		\textbf{Case 2.2.2.} If $p_1=2, \alpha \geq 2$ and $\beta$ is any positive integer, then $K$ contains a subgroup isomorphic to $\mathbb{Z}_{2^2} \times \mathbb{Z}_2 \times \mathbb{Z}_{p_2}$. Consider the hyperedges $e_1, e_2$ and $e_3$, where $e_1 = \{<(0,1,0)>, <(2,1,0)>,<(0,0,1)>,<(1,0,0)>\}, e_2=\{<(0,1,0)>, <(2,1,0)>,<(0,0,1)>,<(1,1,0)>\}$ and $e_3=\{<(0,1,0)>, <(2,1,0)>,<(0,0,1)>,<(2,0,0)>\}$. Clearly, $\mathcal{I}(\tilde{\Gamma}_{\mathcal{H}}(\mathbb{Z}_{2^2} \times \mathbb{Z}_2 \times \mathbb{Z}_{p_2}))$ contains a $K_{3,3}$ structure and, consequently, by Theorem \ref{kurplanar}, Theorem \ref{planarhyper} and Lemma \ref{injgenus}, $\tilde{\Gamma}_{\mathcal{H}}(K)$ is not planar.\\
		\textbf{Case 2.2.3.} If $p_1 \geq 3$ and $\alpha$ is any positive integer, then $K$ contains a subgroup isomorphic to $\mathbb{Z}_{p_1} \times \mathbb{Z}_{p_1} \times \mathbb{Z}_{p_2}$. Let $H_1=<(0,1,0)>, H_2=<(1,0,0)>, H_3=<(1,1,0)>, H_4=<(2,1,0)>,$ $ H=<(1,0,0),(0,1,0)>$ and $H_5=<(0,0,1)>$ be subgroups of $\mathbb{Z}_{p_1} \times \mathbb{Z}_{p_1} \times \mathbb{Z}_{p_2}$. Thus, there are hyperedges $e_1 \supseteq \{H_1,H_2,H_3,H_5\}, e_2 \supseteq \{H_1,H_2,H_3,H_4H_5\}, e_3 \supseteq \{H_2,H_3,H_4,H_1H_5\}$ and $e_4 \supseteq \{H_1,H_3,H_4,H_2H_5\}$. Hence,  $I(\tilde{\Gamma}_\mathcal{H}(G))$ contains a $K_{3,3}$ subdivison and consequently, by Theorem \ref{kurplanar} and Theorem \ref{planarhyper}, $\tilde{\Gamma}_\mathcal{H}(\mathbb{Z}_{p_1} \times \mathbb{Z}_{p_1} \times \mathbb{Z}_{p_2})$ is not planar.\\
		\textbf{Case 2.3.} If $K \cong \mathbb{Z}_{p_1} \times \mathbb{Z}_{p_1} \times \mathbb{Z}_{p_2} \times \mathbb{Z}_{p_3}$, then consider the following cases:\\
		\textbf{Case 2.3.1.} If $p_1=2$, then $K \cong \mathbb{Z}_2 \times \mathbb{Z}_{2{p_2}{p_3}}$, then  consider the hyperedges $e_1,e_2,e_3$ and $e_4$ of $\tilde{\Gamma}_\mathcal{H}(G)$ such that $e_1=\{<(1,0)>, <(1,15)>, <(0,10)>, <(0,6)>\},$ $e_2 = \{<(1,0)>, <(1,15)>, <(0,10)>, <(0,3)>\}, e_3 = \{<(1,0)>, <(0,15)>, <(0,10)>, <(1,3)>\}$ and $e_4 = \{<(0,15)>,<(1,15)>,<(0,10)>, <(1,6)>\}$.  Thus, $\mathcal{I}(\tilde{\Gamma}_\mathcal{H}(\mathbb{Z}_2 \times \mathbb{Z}_{2{p_2}{p_3}}))$ contains a $K_{3,3}$ subdivision and hence, $\tilde{\Gamma}_\mathcal{H}(K)$ is not planar.\\
		\textbf{Case 2.3.2.} If $p_1 \geq 3$, then the proof follows from Case 2.2.2.\\
		\textbf{Case 3.}  Three $p_i'$s are same\\
		\textbf{Case 3.1.} If $K \cong \mathbb{Z}_{p_1} \times \mathbb{Z}_{p_1} \times \mathbb{Z}_{p_1}$, then $\tilde{\Gamma}_\mathcal{H}(K)$ contains a subhypergraph isomorphic to $\tilde{\Gamma}_\mathcal{H}(\mathbb{Z}_2 \times \mathbb{Z}_2 \times \mathbb{Z}_2)$. Consider the hyperedges $e_1,e_2,e_3$ and $e_4$ of $\tilde{\Gamma}_\mathcal{H}(\mathbb{Z}_2 \times \mathbb{Z}_2 \times \mathbb{Z}_2)$ such that $e_1 \supseteq \{<(1,1,1)>,<(0,1,1)>,<(0,0,1)>,<(1,0,0)>\}, e_2 \supseteq \{<(1,1,1)>,<(0,1,1)>,<(0,0,1)>,<(1,0,0),(1,1,0),(0,1,0)>\}, e_3 \supseteq \{<(1,1,1)>,<(0,1,1)>,<(0,1,0)>,<(1,1,0)>,<(1,0,0),(1,0,1),(0,0,1)>\}$ and $e_4 \supseteq \{<(1,0,0)>,<(1,1,1)>,<(0,1,0)>,<(0,0,1)>,<(1,1,0),(1,0,1)(0,1,1)>\}$. Thus, $\mathcal{I}(\tilde{\Gamma}_\mathcal{H}(K))$ contains a $K_{3,3}$ subdivision and hence, by Theorem \ref{kurplanar} and Theorem \ref{planarhyper}, $\tilde{\Gamma}_\mathcal{H}(K)$ is not planar.\\
	If $G$ is not isomorphic to $\mathbb{Z}_{p_1} \times \mathbb{Z}_{p_1}$, $\mathbb{Z}_{2^\alpha} \times \mathbb{Z}_2$ or $\mathbb{Z}_2 \times \mathbb{Z}_2 \times \mathbb{Z}_{p_1^\beta}$ with $\beta =1,2$, then $G$ contains a subgroup isomorphic to one of the above groups. Hence, if $\tilde{\Gamma}_\mathcal{H}(G)$ is planar, then $G$ is isomorphic to either $\mathbb{Z}_{p_1} \times \mathbb{Z}_{p_1}$, $\mathbb{Z}_{2^\alpha} \times \mathbb{Z}_2$ or $\mathbb{Z}_2 \times \mathbb{Z}_2 \times \mathbb{Z}_{p_1^\beta}$ with $\beta =1,2$.
\end{proof}
 \noindent We essentially need the following results on the orientable genus of a hypergraph in order to characterize  $\tilde{\Gamma}_\mathcal{H}(\mathbb{Z}_n)$  that are embeddable on the torus.
\begin{theorem}\label{incidencegraph}\cite{walsh1975hypermaps}
	For any hypergraph $\mathcal{H}$, $g(\mathcal{H})=g(\mathcal{I}(\mathcal{H}))$.
\end{theorem}
\begin{lemma}\cite{bojan}\label{genus} The orientable genus of a complete bi-partite graph is given by:
	 $$g(K_{m,n})= \left\lceil \frac{(m-2)(n-2)}{4}\right\rceil, m,n \geq 2$$
		
\end{lemma}
\begin{theorem}\cite{sb3}\label{genuscyclic}
	The orientable genus, $g(\tilde{\Gamma}_\mathcal{H}(\mathbb{Z}_n))$, of $\tilde{\Gamma}_\mathcal{H}(\mathbb{Z}_n)$ is equal to 1  if and only if $n =p_1^{\alpha}p_2^3$ with $ 3 \leq \alpha \leq 6$ or   $n=p_1^4p_2^4$ or  $n=p_1^{\beta}p_2p_3$ with $3 \leq \beta \leq 5$, where $p_1,p_2$ and $p_3$ are distinct primes.
	
\end{theorem}
\begin{theorem}\label{orgenus}
		The orientable genus, $g(\tilde{\Gamma}_\mathcal{H}(G))$ of $\tilde{\Gamma}_\mathcal{H}(G)$ is equal to 1  if and only if $G \cong \mathbb{Z}_{3^2} \times \mathbb{Z}_3$.
\end{theorem}
\begin{proof}
	If $G \cong \mathbb{Z}_{3^2} \times \mathbb{Z}_3$,  then Figure \ref{torus1}(a) depicts an embedding of $\mathcal{I}(\tilde{\Gamma}_\mathcal{H}(\mathbb{Z}_{3^2} \times \mathbb{Z}_3))$ on a torus and hence,  Theorem \ref{planar} implies that  $g(\tilde{\Gamma}_\mathcal{H}(G))=1$.\\
	Conversely, assume that $g(\mathcal{I}(\tilde{\Gamma}_\mathcal{H}(G)))$ is equal to 1. Then, by Theorem \ref{planar}, $G$ is not isomorphic to one of $\mathbb{Z}_{p_1} \times \mathbb{Z}_{p_1}$, $\mathbb{Z}_{2^\alpha} \times \mathbb{Z}_2$ or $\mathbb{Z}_2 \times \mathbb{Z}_2 \times \mathbb{Z}_{p_1^\beta}$ with $\beta =1,2$. By Lemma \ref{injgenus}, we have, if $K \leq G$ and $g(\tilde{\Gamma}_\mathcal{H}(K)) > 1$, then $g(\tilde{\Gamma}_\mathcal{H}(G)) > 1$. Let $K \cong \mathbb{Z}_{p_1^{\alpha_1}} \times \mathbb{Z}_{p_2^{\alpha_2}} \times \mathbb{Z}_{p_3^{\alpha_3}} \times \cdots \times \mathbb{Z}_{p_n^{\alpha_n}}\leq G$, where $p_i'$s are prime numbers and $\alpha_i'$s are non-negative integers. We consider different possibilities of $K$ with $g(\tilde{\Gamma}_\mathcal{H}(K)) > 1$ case by case.\\ 
	\textbf{Case 1.} All $p_i's$ are distinct\\
	\textbf{Case 1.1.}  If $K \cong \mathbb{Z}_{p_1^{\alpha_1}} \times \mathbb{Z}_{p_2^{\alpha_2}}$, with either $\alpha_1 \geq 7, \alpha_2 = 3$ or $\alpha_1 \geq 5, \alpha_2 = 4$, then, by Theorem 4.4 from \cite{sb3},  $g(\tilde{\Gamma}_\mathcal{H}(K)) > 1$.\\
	\textbf{Case 1.2.} If $K \cong \mathbb{Z}_{p_1^{\alpha_1}} \times \mathbb{Z}_{p_2} \times \mathbb{Z}_{p_3}$ with $\alpha_1 \geq 7$ or $K \cong \mathbb{Z}_{p_1^2} \times \mathbb{Z}_{p_2^2} \times \mathbb{Z}_{p_3}$, then, by Theorem 4.4 from \cite{sb3},  $g(\tilde{\Gamma}_\mathcal{H}(K)) > 1$.\\
	\textbf{Case 1.3.} If $\omega(n) \geq 4$, then, by Theorem \ref{genuscyclic}, $g(\tilde{\Gamma}_\mathcal{H}(K)) > 1$.\\
	\textbf{Case 2.} Two $p_i's$ are same\\
	\textbf{Case 2.1.}  If $K \cong \mathbb{Z}_{p_1^{\alpha_1}} \times \mathbb{Z}_{p_1^{\alpha_2}}$, then consider the following cases:\\
	\textbf{Case 2.1.1.} If either $p_1=2$ or $p_1=3$ and $\alpha_1, \alpha_2 \geq 2$, then $\tilde{\Gamma}_\mathcal{H}(\mathbb{Z}_{2^2} \times \mathbb{Z}_{2^2})$ is a subhypergraph of $\tilde{\Gamma}_\mathcal{H}(K)$ and  Figure \ref{torus1}(b) depicts a subgraph of $\mathcal{I}(\tilde{\Gamma}_\mathcal{H}(\mathbb{Z}_{2^2} \times \mathbb{Z}_{2^2}))$. Clearly, it is not embeddable on a torus and hence, by Theorem \ref{incidencegraph}, $g(\tilde{\Gamma}_\mathcal{H}(K)) > 1$.\\
	\textbf{Case 2.1.2.} If $p_1 \geq 5$, then $ \mathbb{Z}_{p_1^2} \times \mathbb{Z}_{p_1}$, where $p_1 \geq 5$ is a prime, is a subgroup of $K$. Consider the hyperedges  $e_1,e_2,e_3,e_4,e_5$ and $e_6$ with  $e_1 \supseteq \{<(0,1)>, <(p,1)>,<(p,2)>,<(p,3),<(p,4)>,<(p,0)>\}, e_2 \supseteq \{<(0,1)>, <(p,1)>,<(p,2)>,<(p,3),<(p,4)>,<(1,0)>\}, e_3 \supseteq \{<(0,1)>, <(p,1)>,<(p,2)>,<(p,3),<(p,4)>,<(1,1)>\}, e_4 \supseteq \{<(0,1)>, <(p,1)>,<(p,2)>,<(p,3),<(p,4)>,<(1,2)>\}, e_5 \supseteq \{<(0,1)>, <(p,1)>,<(p,2)>,<(p,3),<(p,4)>,<(1,3)>\}$ and $e_6 \supseteq \{<(0,1)>, <(p,1)>,<(p,2)>,<(p,3),<(p,4)>,<(1,4)>\}$. Thus,  $\mathcal{I}(\tilde{\Gamma}_\mathcal{H}(K))$ contains a torus obstruction as shown in Figure \ref{torus1}(c) and consequently,  $g(\tilde{\Gamma}_\mathcal{H}(K)) > 1$.\\
	\textbf{Case 2.2.}  If $K \cong \mathbb{Z}_{p_1} \times \mathbb{Z}_{p_1} \times \mathbb{Z}_{p_2^\alpha}$, then consider the following cases:\\
	\textbf{Case 2.2.1.} If $p_1=2$ and $\alpha \geq 3$, then $\mathbb{Z}_2 \times \mathbb{Z}_2 \times \mathbb{Z}_{p_2^3}$ is a subgroup of $K$.  Figure \ref{torus1}(d) depicts a drawing of $\mathcal{I}(\tilde{\Gamma}_\mathcal{H}( \mathbb{Z}_2 \times \mathbb{Z}_2 \times \mathbb{Z}_{p_2^3}))$ on a torus, where $H_1=<(1,0,0)>, H_2=<(0,1,0)>, H_3=<(1,1,0)>, H=\mathbb{Z}_2 \times \mathbb{Z}_2 \times \{0\}, K_1=<(0,0,p_2^2)>, K_2=<(0,0,p_2)>$ and $K_3=<(0,0,1)>$. Clearly, $\mathcal{I}(\tilde{\Gamma}_\mathcal{H}( \mathbb{Z}_2 \times \mathbb{Z}_2 \times \mathbb{Z}_{p_2^3}))$ is not embeddable on a torus and hence, by Theorem \ref{incidencegraph}, $g(\tilde{\Gamma}_\mathcal{H}(K)) > 1$.\\
			\begin{figure}
		[htbp!]
		\centering
		\subfloat[Embedding of $\mathcal{I}(\tilde{\Gamma}_\mathcal{H}(\mathbb{Z}_{3^2} \times \mathbb{Z}_3))$ on a torus]{\begin{tikzpicture}[scale=.5, every node/.style={circle, fill=black, minimum size=4pt, inner sep=0pt}]
				\draw[thick] (-6,-5) rectangle (3,4); 
				

				\node[label=below:\tiny{${\langle(0,1)\rangle}$}] (u1) at (-1,-5) {};
				\node[label=above:\tiny{${\langle(0,1)\rangle}$}] (u2) at (-1,4) {};
				\node[label=right:\tiny{${\langle(3,1)\rangle}$}] (u3) at (3,0) {};
				\node[label=left:\tiny{${\langle(3,1) \rangle}$}] (u4) at (-6,0) {};
				\node[label=left:\tiny{${\langle(3,2)\rangle}$}] (u5) at (-3,0) {};
				\node[label=above:\tiny{${\langle(3,0)\rangle}$}] (u6) at (-4,2) {};
				\node[label=above:\tiny{${\langle(1,0)\rangle}$}] (u7) at (1,2) {};
				\node[label=left:\tiny{${\langle(1,1)\rangle}$}] (u8) at (-4,-3) {};
				\node[label=right:\tiny{${\langle(1,2)\rangle}$}] (u9) at (1,-3) {};
				
				\node[label=above:\tiny{${e_1}$}] (v1) at (-2.5,1.5) {};
				\node[label=below:\tiny{${e_2}$}] (v2) at (-1,0) {};
				\node[label=below:\tiny{${e_3}$}] (v3) at (-3,-3) {};
				\node[label=above:\tiny{${e_4}$}] (v4) at (0,-2) {};
				\draw (u2) -- (v1);
				\draw (u4) -- (v1);
				\draw (u5) -- (v1);
				\draw (u6) -- (v1);
				\draw (u2) -- (v2);
				\draw (u3) -- (v2);
				\draw (u5) -- (v2);
				\draw (u7) -- (v2);
				\draw (u1) -- (v3);
				\draw (u4) -- (v3);
				\draw (u5) -- (v3);
				\draw (u8) -- (v3);
				\draw (u5) -- (v4);
				\draw (u3) -- (v4);
				\draw (u1) -- (v4);
				\draw (u9) -- (v4);
				
				
				
				
				\draw[->,] (-2,4) -- (-1,4) [right] {};
				\draw[->,] (-2,-5) -- (-1,-5) [right] {};
				\draw[->,] (3,0) -- (3,1) [right] {};
				\draw[->,] (-6,0) -- (-6,1) [right] {};
		\end{tikzpicture}}
		\hspace{.2cm}
		\subfloat[A subgraph of  $\mathcal{I}(\tilde{\Gamma}_\mathcal{H}(\mathbb{Z}_{2^2} \times \mathbb{Z}_{2^2}))$ on a torus]{\begin{tikzpicture}[scale=.5, every node/.style={circle, fill=black, minimum size=4pt, inner sep=0pt}]
				\draw[thick] (-6,-5) rectangle (5,4); 
				

				\node[label=above:\tiny{${\langle(0,2)\rangle}$}] (u1) at (-4,4) {};
				\node[label=below:\tiny{${\langle(0,2)\rangle}$}] (u2) at (-4,-5) {};
				\node[label=above:\tiny{${\langle(2,0)\rangle}$}] (u3) at (-.5,4) {};
				\node[label=below:\tiny{${\langle(2,0)\rangle}$}] (u4) at (-.5,-5) {};
				\node[label=above:\tiny{${\langle(2,2)\rangle}$}] (u5) at (2,4) {};
				\node[label=below:\tiny{${\langle(2,2)\rangle}$}] (u6) at (2,-5) {};
				\node[label=left:\tiny{${\langle(1,-1)\rangle}$}] (u8) at (0,-2) {};
				\node[label=right:\tiny{${\langle(3,-1)\rangle}$}] (u9) at (-3,-4) {};
				\node[label=right:\tiny{${\langle(1,0)\rangle}$}] (u10) at (-.2,0.4) {};
				\node[label=right:\tiny{${\langle(2,-1)\rangle}$}] (u11) at (0,3.5) {};
				
				\node[label=left:\tiny{${e_1}$}] (v1) at (-2.5,1) {};
				\node[label=left:\tiny{${e_2}$}] (v2) at (-4,0) {};
				\node[label=above:\tiny{${e_3}$}] (v3) at (-3,-2) {};
				\node[label=left:\tiny{${e_4}$}] (v4) at (-1,0) {};
				\node[label=left:\tiny{${e_5}$}] (v5) at (-1,-1) {};
				\node[label=left:\tiny{${e_7}$}] (v7) at (-1,2.5) {};
				\node[label=right:\tiny{${e_{10}}$}] (v10) at (4,-2.5) {};

				\draw (u1) -- (v1);
				\draw (u3) -- (v1);
				\draw (u5) -- (v1);
				\draw (u2) -- (v2);
				\draw (u4) -- (v2);
				\draw (u8) -- (v2);
				\draw (u2) -- (v3);
				\draw (u4) -- (v3);
				\draw (u9) -- (v3);
				\draw (u1) to[out=270, in=160](v4);
				\draw (u5) -- (v4);
				\draw (u10) -- (v4);
				\draw (u1) to[out=240, in=160](v5);
				\draw (u5) to[out=300, in=30](v5);
				\draw (u3) -- (v7);
				\draw (u5) -- (v7);
				\draw (u11) -- (v7);
				\draw (u10) -- (v10);
				\draw (u9) -- (v10);
				\draw (u11) -- (v10);

				
				
				\draw[->,] (-2,4) -- (-1,4) [right] {};
				\draw[->,] (-2,-5) -- (-1,-5) [right] {};
				\draw[->,] (5,0) -- (5,1) [right] {};
				\draw[->,] (-6,0) -- (-6,1) [right] {};
		\end{tikzpicture}}
		\hspace{.2cm}
		\subfloat[Torus obstruction in $\mathcal{I}(\tilde{\Gamma}_\mathcal{H}(\mathbb{Z}_{p_1^2} \times \mathbb{Z}_{p_1}))$]{\begin{tikzpicture}[scale=1, rotate=90,  every node/.style={circle, fill=black, minimum size=6pt, inner sep=0pt}]
				
				\node[label=left:{\tiny{$\langle(p,1)\rangle$}}] (u1) at (0,4) {};
				\node[label=right:{\tiny{$\langle(0,1)\rangle$}}] (u2) at (-0.5,1) {};
				\node[label=right:{\tiny{$\langle(p,2)\rangle$}}] (u3) at (0,-2) {};
				\node[label=left:{\tiny{$\langle(p,3)\rangle$}}] (u4) at (-.5,4) {};
				\node[label=right:{\tiny{$\langle(p,4)\rangle$}}] (u5) at (-.5,-2) {};
				\node[label=above:\tiny{$e_1$}] (v1) at (1,4) {};
				\node[label=above:\tiny{$e_2$}] (v2) at (1,1) {};
				\node[label=above:\tiny{$e_3$}] (v3) at (1,-2) {};
				\node[label=below:\tiny{$e_4$}] (v4) at (-1.5,4) {};
				\node[label=below:\tiny{$e_5$}] (v5) at (-1.5,1) {};
				\node[label=below:\tiny{$e_6$}] (v6) at (-1.5,-2) {};
				\draw (u1) -- (v1);
				\draw (u1) -- (v2);
				\draw (u1) -- (v3);
				\draw (u2) -- (v1);
				\draw (u2) -- (v2);
				\draw (u2) -- (v3);
				\draw (u3) -- (v3);
				\draw (u3) -- (v2);
				\draw (u3) -- (v1);
				\draw (u4) -- (v4);
				\draw (u4) -- (v5);
				\draw (u4) -- (v6);
				\draw (u5) -- (v4);
				\draw (u5) -- (v5);
				\draw (u5) -- (v6);
				\draw (u2) -- (v4);
				\draw (u2) -- (v5);
				\draw (u2) -- (v6);
		\end{tikzpicture}}
		\hspace{.2cm}
		\subfloat[$\mathcal{I}(\tilde{\Gamma}_\mathcal{H}(\mathbb{Z}_{2} \times \mathbb{Z}_2 \times \mathbb{Z}_{p_2^3}))$ on a torus]{\begin{tikzpicture}[scale=.7, every node/.style={circle, fill=black, minimum size=4pt, inner sep=0pt}]
				\draw[thick] (-6,-5) rectangle (7,4); 
				

				\node[label=below:\tiny{${H_1}$}] (u1) at (1,-5) {};
				\node[label=above:\tiny{${H_1}$}] (u2) at (1,4) {};
				\node[label=right:\tiny{${H_2}$}] (u3) at (7,0) {};
				\node[label=left:\tiny{${H_2}$}] (u4) at (-6,0) {};
				\node[label=above:\tiny{${H_3}$}] (u5) at (-2,4) {};
				\node[label=above:\tiny{${H_3}$}] (u6) at (-2,-5) {};
				\node[label=above:\tiny{${K_1}$}] (u7) at (-1.5,0) {};
				\node[label=right:\tiny{${K_2}$}] (u8) at (-.8,-1) {};
				\node[label=above:\tiny{${K_3}$}] (u9) at (-2,-2) {};
				\node[label=right:\tiny{${H_3K_1}$}] (u10) at (3,2) {};
				\node[label=right:\tiny{${H_3K_2}$}] (u11) at (3,-2) {};
				\node[label=right:\tiny{${H_3K_3}$}] (u12) at (5,2) {};
				\node[label=right:\tiny{${H_2K_1}$}] (u13) at (0,3) {};
				\node[label=right:\tiny{${H_2K_2}$}] (u14) at (0,2) {};
				\node[label=right:\tiny{${H_2K_3}$}] (u15) at (1,1) {};
				\node[label=right:\tiny{${H_1K_1}$}] (u16) at (-5,3.5) {};
				\node[label=right:\tiny{${H_1K_2}$}] (u17) at (-3.5,2) {};
				\node[label=right:\tiny{${H_1K_3}$}] (u18) at (-5.5,-4) {};
				\node[label=above:\tiny{${H}$}] (u19) at (-.8,0.5) {};
				
				\node[label=right:\tiny{${e_1}$}] (v1) at (-2.5,1) {};
				\node[label=left:\tiny{${e_2}$}] (v2) at (0,0) {};
				\node[label=below:\tiny{${e_3}$}] (v3) at (-3,-3) {};
				\node[label=left:\tiny{${e_4}$}] (v4) at (2,2) {};
				\node[label=above:\tiny{${e_5}$}] (v5) at (2,-2) {};
				\node[label=above:\tiny{${e_6}$}] (v6) at (3,3) {};
				\node[label=above:\tiny{${e_7}$}] (v7) at (-1,3) {};
				\node[label=above:\tiny{${e_8}$}] (v8) at (-.5,2) {};
				\node[label=below:\tiny{${e_9}$}] (v9) at (0,1) {};
				\node[label=left:\tiny{${e_{10}}$}] (v10) at (-5,2.5) {};
				\node[label=left:\tiny{${e_{11}}$}] (v11) at (-4.5,2) {};
				\node[label=left:\tiny{${e_{12}}$}] (v12) at (-5,-2) {};
				\node[label=right:\tiny{${e_{13}}$}] (v13) at (-2,1.5) {};
				\node[label=left:\tiny{${e_{14}}$}] (v14) at (-1,-.5) {};
				\node[label=left:\tiny{${e_{15}}$}] (v15) at (-3,-1.8) {};
				
				\draw (u1) -- (v1);
				\draw (u4) -- (v1);
				\draw (u5) -- (v1);
				\draw (u7) -- (v1);
				\draw (u1) -- (v2);
				\draw (u3) -- (v2);
				\draw (u5) -- (v2);
				\draw (u8) -- (v2);
				\draw (u1) -- (v3);
				\draw (u4) -- (v3);
				\draw (u6) -- (v3);
				\draw (u9) -- (v3);
				\draw (u2) -- (v4);
				\draw (u3) -- (v4);
				\draw (u10) -- (v4);
				\draw (u1) -- (v5);
				\draw (u3) -- (v5);
				\draw (u11) -- (v5);
				\draw (u2) -- (v6);
				\draw (u3) -- (v6);
				\draw (u12) -- (v6);
				\draw (u2) -- (v7);
				\draw (u5) -- (v7);
				\draw (u13) -- (v7);
				\draw (u2) -- (v8);
				\draw (u5) -- (v8);
				\draw (u14) -- (v8);
				\draw (u2) -- (v9);
				\draw (u5) -- (v9);
				\draw (u15) -- (v9);
				\draw (u4) -- (v10);
				\draw (u5) -- (v10);
				\draw (u16) -- (v10);
				\draw (u4) -- (v11);
				\draw (u5) -- (v11);
				\draw (u17) -- (v11);
				\draw (u4) -- (v12);
				\draw (u6) -- (v12);
				\draw (u18) -- (v12);
				\draw (u7) -- (v13);
				\draw (u19) -- (v13);
				\draw (u8) -- (v14);
				\draw (u19) -- (v14);
				\draw (u9) -- (v15);
				\draw (u19) -- (v15);
				
				
				
				\draw[->,] (-2,4) -- (-1,4) [right] {};
				\draw[->,] (-2,-5) -- (-1,-5) [right] {};
				\draw[->,] (7,0) -- (7,1) [right] {};
				\draw[->,] (-6,0) -- (-6,1) [right] {};
		\end{tikzpicture}}
		\caption{}
		\label{torus1}
	\end{figure}
	\textbf{Case 2.2.2.} If $p_1 \geq 3$ and $\alpha$ is a positive integer, then $\mathbb{Z}_{p_1} \times \mathbb{Z}_{p_1} \times \mathbb{Z}_{p_2}$ is a subgroup of $K$. Figure \ref{torus2}(a) depicts a subgroup of $\mathcal{I}(\tilde{\Gamma}_\mathcal{H}(\mathbb{Z}_{p_1} \times \mathbb{Z}_{p_1} \times \mathbb{Z}_{p_2}))$ on a torus, where $H_1 =<(0,1,0)>, H_2=<(1,0,0)>, H_3=<(1,1,0)>, H=\mathbb{Z}_{p_1} \times \mathbb{Z}_{p_1} \times \{0\}, H'=\{0\} \times \{0\} \times \mathbb{Z}_{p_2}$. Clearly, $\mathcal{I}(\tilde{\Gamma}_\mathcal{H}(\mathbb{Z}_{p_1} \times \mathbb{Z}_{p_1} \times \mathbb{Z}_{p_2}))$ is not embeddable on a torus and hence, by Theorem \ref{incidencegraph},  $g(\tilde{\Gamma}_{\mathcal{H}}(K)) > 1$.\\
	\textbf{Case 2.3.} If $K \cong \mathbb{Z}_{p_1} \times \mathbb{Z}_{p_1} \times \mathbb{Z}_{p_2} \times \mathbb{Z}_{p_3}$, then Figure \ref{torus2}(b) depicts a subgraph of $\mathcal{I}(\tilde{\Gamma}_\mathcal{H}(K))$. Clearly, it is not embeddable on a torus and hence, by Theorem \ref{incidencegraph}, $g(\tilde{\Gamma}_\mathcal{H}(K)) > 1$.\\
	\textbf{Case 3}  Three $p_i'$s are same\\
	\textbf{Case 3.1.} If $ K \cong \mathbb{Z}_{p_1^{\alpha_1}} \times \mathbb{Z}_{p_1^{\alpha_2}} \times \mathbb{Z}_{p_1^{\alpha_3}}$, then  $\tilde{\Gamma}_\mathcal{H}(\mathbb{Z}_2 \times \mathbb{Z}_2 \times \mathbb{Z}_2)$ is a subhypergraph of $\tilde{\Gamma}_\mathcal{H}(K)$. Figure \ref{torus2}(c) depicts a drawing of $\mathcal{I}(\tilde{\Gamma}_{\mathcal{H}}(\mathbb{Z}_2 \times \mathbb{Z}_2 \times \mathbb{Z}_2))$ on a torus. Clearly, $\mathcal{I}(\tilde{\Gamma}_{\mathcal{H}}(\mathbb{Z}_2 \times \mathbb{Z}_2 \times \mathbb{Z}_2))$ is not embeddable on a torus and hence, $g(\tilde{\Gamma}_\mathcal{H}(K) > 1$.\\
	If $G$ is not isomorphic to $\mathbb{Z}_{p_1} \times \mathbb{Z}_{p_1}$, $\mathbb{Z}_{2^\alpha} \times \mathbb{Z}_2$ or $\mathbb{Z}_2 \times \mathbb{Z}_2 \times \mathbb{Z}_{p_1^\beta}$ with $\beta =1,2$ or $\mathbb{Z}_{3^2} \times \mathbb{Z}_3$, then $G$ contains a subgroup isomorphic to one of the above groups.  Hence, if $g(\tilde{\Gamma}_\mathcal{H}(G))$  is equal to 1, then $G \cong \mathbb{Z}_{3^2} \times \mathbb{Z}_3$.
	\begin{figure}[htbp!]
		\centering
		\subfloat[$\mathcal{I}(\tilde{\Gamma}_\mathcal{H}(\mathbb{Z}_{3} \times \mathbb{Z}_3 \times \mathbb{Z}_{p_2}))$ on a torus]{\begin{tikzpicture}[scale=.6, every node/.style={circle, fill=black, minimum size=4pt, inner sep=0pt}]
				\draw[thick] (-6,-5) rectangle (4,4); 
				

				\node[label=below:\tiny{${H_1}$}] (u1) at (1,-5) {};
				\node[label=above:\tiny{${H_1}$}] (u2) at (1,4) {};
				\node[label=right:\tiny{${H_2}$}] (u3) at (4,0) {};
				\node[label=left:\tiny{${H_2}$}] (u4) at (-6,0) {};
				\node[label=above:\tiny{${H_3}$}] (u5) at (-1,4) {};
				\node[label=left:\tiny{${H_4}$}] (u6) at (-6,2) {};
				\node[label=above:\tiny{${H_1 H^{'}}$}] (u7) at (1,3) {};
				\node[label=left:\tiny{${H_2H^{'}}$}] (u8) at (-4,-3) {};
				\node[label=right:\tiny{${H_3H^{'}}$}] (u9) at (2,-3) {};
				\node[label=right:\tiny{${H_4}$}] (u11) at (4,2) {};
				\node[label=above:\tiny{${H^{'}}$}] (u12) at (-3,3) {};
				\node[label=below:\tiny{${H_3}$}] (u13) at (-1,-5) {};
				\node[label=below:\tiny{${H_4H^{'}}$}] (u14) at (-3,0) {};
				\node[label=left:\tiny{${H}$}] (u15) at (-4,2) {};
				
				\node[label=right:\tiny{${e_1}$}] (v1) at (-2.5,1) {};
				\node[label=below:\tiny{${e_2}$}] (v2) at (0,0) {};
				\node[label=below:\tiny{${e_3}$}] (v3) at (-3,-3) {};
				\node[label=above:\tiny{${e_4}$}] (v4) at (1,-2) {};
				\node[label=below:\tiny{${e_5}$}] (v5) at (-2,-2) {};
				\node[label=above:\tiny{${e_6}$}] (v6) at (-4,3) {};
				
				\draw (u1) -- (v1);
				\draw (u4) -- (v1);
				\draw (u5) -- (v1);
				\draw (u6) -- (v1);
				\draw (u12) -- (v1);
				\draw (u3) -- (v2);
				\draw (u5) -- (v2);
				\draw (u11) -- (v2);
				\draw (u7) -- (v2);
				\draw (u1) -- (v3);
				\draw (u13) -- (v3);
				\draw (u6) -- (v3);
				\draw (u8) -- (v3);
				\draw (u11) -- (v4);
				\draw (u3) -- (v4);
				\draw (u1) -- (v4);
				\draw (u9) -- (v4);
				\draw (u1) -- (v5);
				\draw (u4) -- (v5);
				\draw (u5) -- (v5);
				\draw (u14) -- (v5);
				\draw (u15) -- (v6);
				\draw (u12) -- (v6);
				
				
				
				
				\draw[->,] (-2,4) -- (0,4) [right] {};
				\draw[->,] (-2,-5) -- (0,-5) [right] {};
				\draw[->,] (4,0) -- (4,1) [right] {};
				\draw[->,] (-6,0) -- (-6,1) [right] {};
		\end{tikzpicture}}
	\hspace{.2cm}
		\subfloat[A subgraph of $\mathcal{I}( \mathbb{Z}_{p_1} \times \mathbb{Z}_{p_1} \times \mathbb{Z}_{p_2} \times \mathbb{Z}_{p_3})$ on a torus]{\begin{tikzpicture}[scale=.6, every node/.style={circle, fill=black, minimum size=4pt, inner sep=0pt}]
				\draw[thick] (-6,-5) rectangle (7,4); 
				

				\node[label=below right:\tiny{${\langle(1,0,0,0)\rangle}$}] (u1) at (2,-5) {};
				\node[label=above right:\tiny{${\langle(1,0,0,0)\rangle}$}] (u2) at (2,4) {};
				\node[label=below:\tiny{${\langle(0,1,0,0)\rangle}$}] (u3) at (-1,-5) {};
				\node[label=above:\tiny{${\langle(0,1,0,0)\rangle}$}] (u4) at (-1,4) {};
				\node[label=right:\tiny{${\langle(1,1,0,0)\rangle}$}] (u5) at (7,1) {};
				\node[label=left:\tiny{${\langle(1,1,0,0)\rangle}$}] (u6) at (-6,1) {};
				\node[label=right:\tiny{${\langle(0,0,1,0)\rangle}$}] (u7) at (7,-1) {};
				\node[label=left:\tiny{${\langle(0,0,1,0)\rangle}$}] (u8) at (-6,-1) {};
				\node[label=right:\tiny{${\langle(0,0,1,1)\rangle}$}] (u9) at (1.5,-2) {};
				\node[label=below left:\tiny{${\langle(0,0,0,1)\rangle}$}] (u10) at (-4,-5) {};
				\node[label=above left:\tiny{${\langle(0,0,0,1)\rangle}$}] (u11) at (-4,4) {};
				\node[label=left:\tiny{${\langle(1,1,0,0),(0,0,1,0)\rangle}$}] (u12) at (-4,-2) {};
				\node[label=left:\tiny{${\langle(0,1,0,0),(0,0,1,0)\rangle}$}] (u13) at (4,3) {};
				
				\node[label=below:\tiny{${e_1}$}] (v1) at (-2.5,1) {};
				\node[label=above right:\tiny{${e_2}$}] (v2) at (0,-1) {};
				\node[label=right:\tiny{${e_3}$}] (v3) at (-3,-2) {};
				\node[label=left:\tiny{${e_4}$}] (v4) at (-6,3) {};
				\node[label=right:\tiny{${e_4}$}] (v5) at (7,3) {};
				\node[label=right:\tiny{${e_5}$}] (v6) at (3,1) {};
				
				\draw (u2) -- (v1);
				\draw (u4) -- (v1);
				\draw (u6) -- (v1);
				\draw (u8) -- (v1);
				\draw (u11) -- (v1);
				\draw (u1) -- (v2);
				\draw (u3) -- (v2);
				\draw (u5) -- (v2);
				\draw (u9) -- (v2);
				\draw (u2) -- (v3);
				\draw (u3) -- (v3);
				\draw (u10) -- (v3);
				\draw (u12) -- (v3);
				\draw (u6) to[out=60, in=300](v4);
				\draw (u11) -- (v4);
				\draw (u2) -- (v5);
				\draw (u13) -- (v5);
				\draw (u2) -- (v6);
				\draw (u3) to[out=100, in=170](v6);
				\draw (u7) -- (v6);
				
				
				\draw[->,] (-2,4) -- (0,4) [right] {};
				\draw[->,] (-2,-5) -- (0,-5) [right] {};
				\draw[->,] (7,-1) -- (7,0) [right] {};
				\draw[->,] (-6,-1) -- (-6,0) [right] {};
		\end{tikzpicture}}
	\hspace{.2cm}
		\subfloat[ A subgraph of $\mathcal{I}(\tilde{\Gamma}_\mathcal{H}(\mathbb{Z}_{2} \times \mathbb{Z}_2 \times \mathbb{Z}_2))$ on a torus]{	\begin{tikzpicture}[scale=.5, every node/.style={circle, fill=black, minimum size=4pt, inner sep=0pt}]
				\draw[thick] (-6,-5) rectangle (7,4); 
				

				\node[label=right:\tiny{${\langle(1,0,0)\rangle}$}] (u1) at (2,3) {};
				\node[label=right:\tiny{${\langle(0,1,0)\rangle}$}] (u2) at (-2,3) {};				
				\node[label=left:\tiny{${\langle(0,0,1)\rangle}$}] (u3) at (-3.5,3) {};
				\node[label=left:\tiny{${\langle(1,1,0)\rangle}$}] (u4) at (-6,0) {};
				\node[label=right:\tiny{${\langle(1,1,0)\rangle}$}] (u5) at (7,0) {};
				\node[label=left:\tiny{${\langle(1,0,1)\rangle}$}] (u6) at (2,-3) {};
				\node[label=left:\tiny{${\langle(0,1,1)\rangle}$}] (u7) at (-1,-3) {};
				\node[label=left:\tiny{${\langle(1,1,1)\rangle}$}] (u8) at (-3.5,-3) {};
				
				\node[label=below:\tiny{${e_1}$}] (v1) at (-2,-5) {};
				\node[label=above:\tiny{${e_1}$}] (v2) at (-2,4) {};
				\node[label=left:\tiny{${e_2}$}] (v3) at (-2,0) {};
				\node[label=left:\tiny{${e_3}$}] (v4) at (2,1) {};
				\node[label=left:\tiny{${e_4}$}] (v5) at (4,2) {};
				\node[label=left:\tiny{${e_5}$}] (v6) at (4,-1) {};
				
				\draw (u1) -- (v2);
				\draw (u2) -- (v2);
				\draw (u3) -- (v2);
				\draw (u5) -- (v2);
				\draw (u6) -- (v1);
				\draw (u7) -- (v1);
				\draw (u8) -- (v1);
				\draw (u3) -- (v3);
				\draw (u6) -- (v3);
				\draw (u7) -- (v3);
				\draw (u8) -- (v3);
				\draw (u3) -- (v4);
				\draw (u5) -- (v4);
				\draw (u6) -- (v4);
				\draw (u2) -- (v4);
				\draw (u1) -- (v5);
				\draw (u5) -- (v5);
				\draw (u6) -- (v5);
				\draw (u8) -- (v5);
				\draw (u1) -- (v6);
				\draw (u3) -- (v6);
				\draw (u5) -- (v6);
				\draw (u7) -- (v6);

				\draw[->,] (-2,4) -- (-1,4) [right] {};
				\draw[->,] (-2,-5) -- (-1,-5) [right] {};
				\draw[->,] (7,0) -- (7,1) [right] {};
				\draw[->,] (-6,0) -- (-6,1) [right] {};
		\end{tikzpicture}}
		\caption{}
		\label{torus2}
	\end{figure}
\end{proof}
\begin{corollary}
 $\tilde{\Gamma}_\mathcal{H}(G)$ is toroidal if and only if $G \cong \mathbb{Z}_{p_1} \times \mathbb{Z}_{p_1}$, where $p_1$ is a prime or $ G \cong  Z_{2^\alpha} \times \mathbb{Z}_2$	or  $ G \cong \mathbb{Z}_2 \times \mathbb{Z}_2 \times \mathbb{Z}_{p_2^\beta}$, where $p_2$ is prime and  $\beta =1,2$, or $G \cong \mathbb{Z}_{3^2} \times \mathbb{Z}_3$.
\end{corollary}
We rely on the following results to characterize those hypergraphs 
$\tilde{\Gamma}_\mathcal{H}(G)$ that are embeddable on the projective plane.
\begin{theorem}\label{incidencegraph2}\cite{walsh1975hypermaps}
	For any hypergraph $\mathcal{H}$, $\tilde{g}(\mathcal{H})=\tilde{g}(\mathcal{I}(\mathcal{H}))$.
\end{theorem}
\begin{lemma}\cite{bojan}\label{genus2} The non-orientable genus of a complete bi-partite graph is given by:
 $$\tilde{g}(K_{m,n})= \left\lceil \frac{(m-2)(n-2)}{2}\right\rceil, m,n \geq 2$$
\end{lemma}
	\begin{theorem}\cite{sb3}\label{projcyclic}
	The non-orientable genus, $\tilde{g}(\tilde{\Gamma}_\mathcal{H}(\mathbb{Z}_n))$, of $\tilde{\Gamma}_\mathcal{H}(\mathbb{Z}_n)$ is equal to 1 if and only if  $n=p_1^{\alpha_1}p_2^3$ with $\alpha_1=3,4$ or $n=p_1^{\alpha_1}p_2p_3$ with $\alpha_1=3,4$, where $p_1,p_2$ and $p_3$ are distinct primes.
	
\end{theorem}
\begin{theorem}
	The non-orientable genus, $\tilde{g}(\tilde{\Gamma}_\mathcal{H}(G))$ of $\tilde{\Gamma}_\mathcal{H}(G)$ is equal to 1  if and only if $\mathbb{Z}_{3^2} \times \mathbb{Z}_3$.
\end{theorem}
\begin{proof}

	 If $G \cong \mathbb{Z}_{3^2} \times \mathbb{Z}_3$, then Figure \ref{proj1}(a) depicts an embedding of $\mathcal{I}(\tilde{\Gamma}_\mathcal{H}(\mathbb{Z}_{3^2} \times \mathbb{Z}_3))$  on a projective plane. Hence, Theorem \ref{planar} implies that $\tilde{g}(\tilde{\Gamma}_\mathcal{H}(G)) = 1$.\\Conversely, assume that $\tilde{g}(\mathcal{I}(\tilde{\Gamma}_\mathcal{H}(G)))$ is equal to 1. Then, by Theorem \ref{planar}, $G$ is not isomorphic to one of $\mathbb{Z}_{p_1} \times \mathbb{Z}_{p_1}$, $\mathbb{Z}_{2^\alpha} \times \mathbb{Z}_2$ or $\mathbb{Z}_2 \times \mathbb{Z}_2 \times \mathbb{Z}_{p_1^\beta}$ with $\beta =1,2$. By Lemma \ref{injgenus}, we have, if $K \leq G$ and $\tilde{g}(\tilde{\Gamma}_\mathcal{H}(K)) > 1$, then $\tilde{g}(\tilde{\Gamma}_\mathcal{H}(G)) > 1$. Let $K \cong \mathbb{Z}_{p_1^{\alpha_1}} \times \mathbb{Z}_{p_2^{\alpha_2}} \times \mathbb{Z}_{p_3^{\alpha_3}} \times \cdots \times \mathbb{Z}_{p_n^{\alpha_n}}\leq G$, where $p_i'$s are prime numbers and $\alpha_i'$s are non-negative integers. We consider different possibilities of $K$ with $\tilde{g}(\tilde{\Gamma}_\mathcal{H}(K)) > 1$ case by case.\\
	 	\textbf{Case 1.} All $p_i's$ are distinct\\
	 \textbf{Case 1.1.}  If $K \cong \mathbb{Z}_{p_1^{\alpha_1}} \times \mathbb{Z}_{p_2^{\alpha_2}}$, with either $\alpha_1 \geq 5, \alpha_2 = 3$ or $\alpha_1 \geq 4, \alpha_2 = 4$, then, by Theorem \ref{projcyclic},  $\tilde{g}(\tilde{\Gamma}_\mathcal{H}(K)) > 1$.\\ 
	 \textbf{Case 1.2.} If $K \cong \mathbb{Z}_{p_1^{\alpha_1}} \times \mathbb{Z}_{p_2} \times \mathbb{Z}_{p_3}$ with $\alpha_1 \geq 5$ or $K \cong \mathbb{Z}_{p_1^2} \times \mathbb{Z}_{p_2^2} \times \mathbb{Z}_{p_3}$, then, by Theorem \ref{projcyclic},  $\tilde{g}(\tilde{\Gamma}_\mathcal{H}(K)) > 1$.\\
	 	\textbf{Case 1.3.} If $\omega(n) \geq 4$, then, by Theorem \ref{projcyclic}, $\tilde{g}(\tilde{\Gamma}_\mathcal{H}(K)) > 1$.\\
	 		\textbf{Case 2.} Two $p_i's$ are same\\
	 	\textbf{Case 2.1.}  If $K \cong \mathbb{Z}_{p_1^{\alpha_1}} \times \mathbb{Z}_{p_1^{\alpha_2}}$, then consider the following cases:\\	\textbf{Case 2.1.1.} If either $p_1=2$ or $p_1 =3$ and $\alpha_1, \alpha_2 \geq 2$, then $ \tilde{\Gamma}_\mathcal{H}(\mathbb{Z}_{2^2} \times \mathbb{Z}_{2^2})$ is a subhypergraph of $\tilde{\Gamma}_\mathcal{H}(K)$ and Figure \ref{proj1}(b) depicts a drawing of a subgraph of $\mathcal{I}(\tilde{\Gamma}_\mathcal{H}(\mathbb{Z}_{2^2} \times \mathbb{Z}_{2^2}))$ on a projective plane. Clearly, it is not embeddable on a projective plane and hence, by Theorem \ref{incidencegraph2},  $\tilde{g}(\tilde{\Gamma}_\mathcal{H}(G)) > 1$.\\
	 			\textbf{Case 2.1.2.} If $p_1 \geq 5$, then $ \mathbb{Z}_{p_1^2} \times \mathbb{Z}_{p_1}$, where $p_1 \geq 5$ is a prime, is a subgroup of $K$. Consider the following hyperedges of $ \mathbb{Z}_{p_1^2} \times \mathbb{Z}_{p_1}$: $e_1=\{<(0,1)>,<(p,1)>,<(p,2)>,<(p,0)>\}, e_2=\{<(0,1)>,<(p,1)>,<(p,2)>,<(1,0)>\}, e_3=\{<(0,1)>,<(p,1)>,<(p,2)>,<(1,1)>\}, e_4=\{<(0,1)>,<(p,1)>,<(p,2)>,<(1,2)>\}, e_5=\{<(0,1)>,<(p,1)>,<(p,2)>,<(1,3)>\}$ and  $e_6=\{<(0,1)>,<(p,1)>,<(p,2)>,<(1,4)>\}$. Hence,  $\mathcal{I}(\tilde{\Gamma}_\mathcal{H}(K))$ contains $K_{3,5}$ and consequently, $\tilde{g}(\tilde{\Gamma}_\mathcal{H}(G)) > 1$.\\
	 			\textbf{Case 2.2.}  If $K \cong \mathbb{Z}_{p_1} \times \mathbb{Z}_{p_1} \times \mathbb{Z}_{p_2^\alpha}$, then consider the following cases:\\
	 			\textbf{Case 2.2.1.} If $p_1=2$ and $\alpha \geq 3$, then $\mathbb{Z}_2 \times \mathbb{Z}_2 \times \mathbb{Z}_{p_2^3}$ is a subgroup of $K$.  Figure \ref{proj1}(c) depicts a drawing of $\mathcal{I}(\tilde{\Gamma}_\mathcal{H}(\mathbb{Z}_{2} \times \mathbb{Z}_2 \times \mathbb{Z}_{p_2^3}))$ on a projective plane, where $H_1=<(1,0,0)>, H_2=<(0,1,0)>, H_3=<(1,1,0)>, H=\mathbb{Z}_2 \times \mathbb{Z}_2 \times \{0\}, K_1=<(0,0,p_2^2)>, K_2=<(0,0,p_2)>$ and $K_3=<(0,0,1)>$. Clearly, it is not embeddable on a projective plane and hence, by Theorem \ref{incidencegraph2}, $\tilde{g}(\tilde{\Gamma}_\mathcal{H}(K)) > 1$.\\
	 			  	\begin{figure}[htbp!]
	 				\centering
	 				\subfloat[Embedding of $\mathcal{I}(\tilde{\Gamma}_\mathcal{H}(\mathbb{Z}_{3^2} \times \mathbb{Z}_3))$ on a projective plane]{\begin{tikzpicture}[scale=.5, every node/.style={circle, fill=black, minimum size=4pt, inner sep=0pt}]
	 						\draw[thick] (-6,-4) rectangle (6,4); 
	 						

	 						\node[label=below:\tiny{${\langle(0,1)\rangle}$}] (u1) at (1,-4) {};
	 						\node[label=above:\tiny{${\langle(0,1)\rangle}$}] (u2) at (1,4) {};
	 						\node[label=right:\tiny{${\langle(3,1)\rangle}$}] (u3) at (6,0) {};
	 						\node[label=left:\tiny{${\langle(3,1)\rangle}$}] (u4) at (-6,0) {};
	 						\node[label=left:\tiny{${\langle(3,2)\rangle}$}] (u5) at (-1,0) {};
	 						\node[label=above:\tiny{${\langle(3,0)\rangle}$}] (u6) at (-4,2) {};
	 						\node[label=above:\tiny{${\langle(1,0)\rangle}$}] (u7) at (4,2) {};
	 						\node[label=left:\tiny{${\langle(1,1)\rangle}$}] (u8) at (-4,-3) {};
	 						\node[label=right:\tiny{${\langle(1,2)\rangle}$}] (u9) at (4,-3) {};
	 						
	 						\node[label=above:\tiny{${e_1}$}] (v1) at (-2.5,1) {};
	 						\node[label=below:\tiny{${e_2}$}] (v2) at (0,0) {};
	 						\node[label=below:\tiny{${e_3}$}] (v3) at (-3,-3) {};
	 						\node[label=above:\tiny{${e_4}$}] (v4) at (2,-2) {};
	 						\draw (u2) -- (v1);
	 						\draw (u4) -- (v1);
	 						\draw (u5) -- (v1);
	 						\draw (u6) -- (v1);
	 						\draw (u2) -- (v2);
	 						\draw (u3) -- (v2);
	 						\draw (u5) -- (v2);
	 						\draw (u7) -- (v2);
	 						\draw (u1) -- (v3);
	 						\draw (u4) -- (v3);
	 						\draw (u5) -- (v3);
	 						\draw (u8) -- (v3);
	 						\draw (u5) -- (v4);
	 						\draw (u3) -- (v4);
	 						\draw (u1) -- (v4);
	 						\draw (u9) -- (v4);
	 						
	 						
	 						
	 						
	 						\draw[->,] (-2,4) -- (-1,4) [right] {};
	 						\draw[->,] (-1,-4) -- (-2,-4) [right] {};
	 						\draw[->,] (6,2) -- (6,1) [right] {};
	 						\draw[->,] (-6,0) -- (-6,1) [right] {};
	 				\end{tikzpicture}}
	 				
	 				\subfloat[A subgraph of $\mathcal{I}(\tilde{\Gamma}_\mathcal{H}(\mathbb{Z}_{2^2} \times \mathbb{Z}_{2^2}))$ on a projective plane]{\begin{tikzpicture}[scale=.6, every node/.style={circle, fill=black, minimum size=4pt, inner sep=0pt}]
	 						\draw[thick] (-6,-4) rectangle (4,4); 
	 						

	 						\node[label=above:\tiny{${\langle(0,2)\rangle}$}] (u1) at (-4,4) {};
	 						\node[label=below:\tiny{${\langle(0,2)\rangle}$}] (u2) at (2,-4) {};
	 						\node[label=above:\tiny{${\langle(2,0)\rangle}$}] (u3) at (-.5,4) {};
	 						\node[label=below:\tiny{${\langle(2,0)\rangle}$}] (u4) at (-.5,-4) {};
	 						\node[label=above:\tiny{${\langle(2,2)\rangle}$}] (u5) at (2,4) {};
	 						\node[label=below:\tiny{${\langle(2,2)\rangle}$}] (u6) at (-4,-4) {};
	 						\node[label=left:\tiny{${\langle(1,-1)\rangle}$}] (u8) at (-4,-2) {};
	 						\node[label=right:\tiny{${\langle(3,-1)\rangle}$}] (u9) at (1,-1) {};
	 						\node[label=right:\tiny{${\langle(1,0)\rangle}$}] (u10) at (-.2,0.4) {};
	 						\node[label=right:\tiny{${\langle(2,-1)\rangle}$}] (u11) at (0,3.5) {};
	 						
	 						\node[label=left:\tiny{${e_1}$}] (v1) at (-2.5,1) {};
	 						\node[label=left:\tiny{${e_2}$}] (v2) at (-4,0) {};
	 						\node[label=above:\tiny{${e_3}$}] (v3) at (3,-2) {};
	 						\node[label=left:\tiny{${e_4}$}] (v4) at (-1,0) {};
	 						\node[label=left:\tiny{${e_5}$}] (v5) at (-1,-1) {};
	 						\node[label=above:\tiny{${e_7}$}] (v7) at (-1,2.5) {};
	 						\node[label=right:\tiny{${e_{10}}$}] (v10) at (0,-2.5) {};

	 						\draw (u1) -- (v1);
	 						\draw (u3) -- (v1);
	 						\draw (u5) -- (v1);
	 						\draw (u2) -- (v2);
	 						\draw (u4) -- (v2);
	 						\draw (u8) -- (v2);
	 						\draw (u2) -- (v3);
	 						\draw (u4) -- (v3);
	 						\draw (u9) -- (v3);
	 						\draw (u1) to[out=270, in=160](v4);
	 						\draw (u5) -- (v4);
	 						\draw (u10) -- (v4);
	 						\draw (u1) to[out=240, in=160](v5);
	 						\draw (u5) to[out=300, in=30](v5);
	 						\draw (u3) -- (v7);
	 						\draw (u5) -- (v7);
	 						\draw (u11) -- (v7);
	 						\draw (u10) -- (v10);
	 						\draw (u9) -- (v10);
	 						\draw (u11) -- (v10);

	 						
	 						
	 						\draw[->,] (-2,4) -- (-1,4) [right] {};
	 						\draw[->,] (-2,-4) -- (-3,-4) [right] {};
	 						\draw[->,] (4,1) -- (4,0) [right] {};
	 						\draw[->,] (-6,0) -- (-6,1) [right] {};
	 				\end{tikzpicture}}
	 				\hspace{.2cm}
	 				\subfloat[$\mathcal{I}(\tilde{\Gamma}_\mathcal{H}(\mathbb{Z}_{2} \times \mathbb{Z}_2 \times \mathbb{Z}_{p_2^3}))$ on a projective plane]{\begin{tikzpicture}[scale=.7, every node/.style={circle, fill=black, minimum size=4pt, inner sep=0pt}]
	 						\draw[thick] (-6,-5) rectangle (7,4); 
	 						

	 						\node[label=below:\tiny{${H_1}$}] (u1) at (-2,-5) {};
	 						\node[label=above:\tiny{${H_1}$}] (u2) at (1,4) {};
	 						\node[label=right:\tiny{${H_2}$}] (u3) at (7,0) {};
	 						\node[label=left:\tiny{${H_2}$}] (u4) at (-6,0) {};
	 						\node[label=above:\tiny{${H_3}$}] (u5) at (-2,4) {};
	 						\node[label=below:\tiny{${H_3}$}] (u6) at (1,-5) {};
	 						\node[label=above:\tiny{${K_1}$}] (u7) at (-1.5,0) {};
	 						\node[label=right:\tiny{${K_2}$}] (u8) at (-.8,-1) {};
	 						\node[label=above:\tiny{${K_3}$}] (u9) at (-2,-2) {};
	 						\node[label=right:\tiny{${H_3K_1}$}] (u10) at (3,2) {};
	 						\node[label=right:\tiny{${H_3K_2}$}] (u11) at (3,-2) {};
	 						\node[label=right:\tiny{${H_3K_3}$}] (u12) at (5,2) {};
	 						\node[label=right:\tiny{${H_2K_1}$}] (u13) at (0,3) {};
	 						\node[label=right:\tiny{${H_2K_2}$}] (u14) at (0,2) {};
	 						\node[label=right:\tiny{${H_2K_3}$}] (u15) at (1,1) {};
	 						\node[label=right:\tiny{${H_1K_1}$}] (u16) at (-5,3.5) {};
	 						\node[label=right:\tiny{${H_1K_2}$}] (u17) at (-3.5,2) {};
	 						\node[label=right:\tiny{${H_1K_3}$}] (u18) at (-5.5,-4) {};
	 						\node[label=above:\tiny{${H}$}] (u19) at (-.8,0.5) {};
	 						
	 						\node[label=right:\tiny{${e_1}$}] (v1) at (-2.5,1) {};
	 						\node[label=left:\tiny{${e_2}$}] (v2) at (0,0) {};
	 						\node[label=below:\tiny{${e_3}$}] (v3) at (-3,-3) {};
	 						\node[label=left:\tiny{${e_4}$}] (v4) at (2,2) {};
	 						\node[label=above:\tiny{${e_5}$}] (v5) at (2,-2) {};
	 						\node[label=above:\tiny{${e_6}$}] (v6) at (3,3) {};
	 						\node[label=above:\tiny{${e_7}$}] (v7) at (-1,3) {};
	 						\node[label=above:\tiny{${e_8}$}] (v8) at (-.5,2) {};
	 						\node[label=below:\tiny{${e_9}$}] (v9) at (0,1) {};
	 						\node[label=left:\tiny{${e_{10}}$}] (v10) at (-5,2.5) {};
	 						\node[label=left:\tiny{${e_{11}}$}] (v11) at (-4.5,2) {};
	 						\node[label=left:\tiny{${e_{12}}$}] (v12) at (-5,-2) {};
	 						\node[label=right:\tiny{${e_{13}}$}] (v13) at (-2,1.5) {};
	 						\node[label=left:\tiny{${e_{14}}$}] (v14) at (-1,-.5) {};
	 						\node[label=left:\tiny{${e_{15}}$}] (v15) at (-3,-1.8) {};
	 						
	 						\draw (u1) -- (v1);
	 						\draw (u4) -- (v1);
	 						\draw (u5) -- (v1);
	 						\draw (u7) -- (v1);
	 						\draw (u1) -- (v2);
	 						\draw (u3) -- (v2);
	 						\draw (u5) -- (v2);
	 						\draw (u8) -- (v2);
	 						\draw (u1) -- (v3);
	 						\draw (u4) -- (v3);
	 						\draw (u6) -- (v3);
	 						\draw (u9) -- (v3);
	 						\draw (u2) -- (v4);
	 						\draw (u3) -- (v4);
	 						\draw (u10) -- (v4);
	 						\draw (u1) -- (v5);
	 						\draw (u3) -- (v5);
	 						\draw (u11) -- (v5);
	 						\draw (u2) -- (v6);
	 						\draw (u3) -- (v6);
	 						\draw (u12) -- (v6);
	 						\draw (u2) -- (v7);
	 						\draw (u5) -- (v7);
	 						\draw (u13) -- (v7);
	 						\draw (u2) -- (v8);
	 						\draw (u5) -- (v8);
	 						\draw (u14) -- (v8);
	 						\draw (u2) -- (v9);
	 						\draw (u5) -- (v9);
	 						\draw (u15) -- (v9);
	 						\draw (u4) -- (v10);
	 						\draw (u5) -- (v10);
	 						\draw (u16) -- (v10);
	 						\draw (u4) -- (v11);
	 						\draw (u5) -- (v11);
	 						\draw (u17) -- (v11);
	 						\draw (u4) -- (v12);
	 						\draw (u6) -- (v12);
	 						\draw (u18) -- (v12);
	 						\draw (u7) -- (v13);
	 						\draw (u19) -- (v13);
	 						\draw (u8) -- (v14);
	 						\draw (u19) -- (v14);
	 						\draw (u9) -- (v15);
	 						\draw (u19) -- (v15);
	 						
	 						
	 						
	 						\draw[->,] (-2,4) -- (-1,4) [right] {};
	 						\draw[->,] (1,-5) -- (0,-5) [right] {};
	 						\draw[->,] (7,1) -- (7,-1) [right] {};
	 						\draw[->,] (-6,0) -- (-6,1) [right] {};
	 				\end{tikzpicture}}
	 				\caption{}
	 				\label{proj1}
	 			\end{figure}
	 			\textbf{Case 2.2.2.} If $p_1 \geq 3$, then $\mathbb{Z}_{p_1} \times \mathbb{Z}_{p_1} \times \mathbb{Z}_{p_2}$ is a subgroup of $K$. Figure \ref{proj2}(a) depicts a drawing of $\mathcal{I}(\tilde{\Gamma}_\mathcal{H}(\mathbb{Z}_{p_1} \times \mathbb{Z}_{p_1} \times \mathbb{Z}_{p_2}))$ on a projective plane, where $H_1 =<(0,1,0)>, H_2=<(1,0,0)>, H_3=<(1,1,0)>, H=\mathbb{Z}_{p_1} \times \mathbb{Z}_{p_1} \times \{0\}, H'=\{0\} \times \{0\} \times \mathbb{Z}_{p_2}$. Clearly, it is not embeddable on a projective plane and hence, $\tilde{g}(\tilde{\Gamma}_\mathcal{H}(K)) > 1$.\\
	 				\textbf{Case 2.3.} If $K \cong \mathbb{Z}_{p_1} \times \mathbb{Z}_{p_1} \times \mathbb{Z}_{p_2} \times \mathbb{Z}_{p_3}$, then  Figure \ref{proj2}(b) depicts a drawing of subgraph of  $\mathcal{I}(\tilde{\Gamma}_\mathcal{H}(\mathbb{Z}_{p_1} \times \mathbb{Z}_{p_1} \times \mathbb{Z}_{p_2})$ on a projective plane. Clearly, it is not embeddable on a projective plane and hence, by Theorem \ref{incidencegraph2},  $\tilde{g}(\tilde{\Gamma}_\mathcal{H}(K)) > 1$.\\
	 \textbf{Case 3}  Three $p_i'$s are same\\
	 \textbf{Case 3.1.} If $ K \cong \mathbb{Z}_{p_1^{\alpha_1}} \times \mathbb{Z}_{p_1^{\alpha_2}} \times \mathbb{Z}_{p_1^{\alpha_3}}$, then $\tilde{\Gamma}_\mathcal{H}(\mathbb{Z}_2 \times \mathbb{Z}_2 \times \mathbb{Z}_2)$ is a subhypergraph of $\tilde{\Gamma}_\mathcal{H}(K)$. Figure \ref{proj2}(c) depicts a drawing of $\mathcal{I}(\tilde{\Gamma}_\mathcal{H}(\mathbb{Z}_{2} \times \mathbb{Z}_2 \times \mathbb{Z}_2))$ on a projective plane. Clearly, it is not embeddable on a projective plane and hence, by Theorem \ref{incidencegraph}, $\tilde{g}(\tilde{\Gamma}_\mathcal{H}(K)) > 1$.\\
	  If $G$ is not isomorphic to $\mathbb{Z}_{p_1} \times \mathbb{Z}_{p_1}$, where $p_1$ is a prime or $ Z_{2^\alpha} \times \mathbb{Z}_2$	or  $ \mathbb{Z}_2 \times \mathbb{Z}_2 \times \mathbb{Z}_{p_2^\beta}$, where $p_2$ is prime and  $\beta =1,2$, or $ \mathbb{Z}_{3^2} \times \mathbb{Z}_3$, then $G$ contains a subgroup isomorphic to one of the above groups.  Hence, if $\tilde{g}(\tilde{\Gamma}_\mathcal{H}(G))$  is equal to 1, then $G \cong \mathbb{Z}_{3^2} \times \mathbb{Z}_3$.
	   \begin{figure}[htbp!]
	  	\centering
	  	\subfloat[$\mathcal{I}(\tilde{\Gamma}_\mathcal{H}(\mathbb{Z}_{3} \times \mathbb{Z}_3 \times \mathbb{Z}_{p_2}))$ on a projective plane]{\begin{tikzpicture}[scale=.5, every node/.style={circle, fill=black, minimum size=4pt, inner sep=0pt}]
	  			\draw[thick] (-6,-5) rectangle (7,4); 
	  			

	  			\node[label=below:\tiny{${H_1}$}] (u1) at (-1,-5) {};
	  			\node[label=above:\tiny{${H_1}$}] (u2) at (1,4) {};
	  			\node[label=right:\tiny{${H_2}$}] (u3) at (7,0) {};
	  			\node[label=left:\tiny{${H_2}$}] (u4) at (-6,2) {};
	  			\node[label=above:\tiny{${H_3}$}] (u5) at (-1,4) {};
	  			\node[label=left:\tiny{${H_4}$}] (u6) at (-6,0) {};
	  			\node[label=above:\tiny{${H_1 H^{'}}$}] (u7) at (3,3) {};
	  			\node[label=left:\tiny{${H_2H^{'}}$}] (u8) at (-4,-3) {};
	  			\node[label=right:\tiny{${H_3H^{'}}$}] (u9) at (4,-3) {};
	  			\node[label=right:\tiny{${H_4}$}] (u11) at (7,2) {};
	  			\node[label=above:\tiny{${H^{'}}$}] (u12) at (-3,3) {};
	  			\node[label=below:\tiny{${H_3}$}] (u13) at (1,-5) {};
	  			\node[label=below:\tiny{${H_4H^{'}}$}] (u14) at (-3,0) {};
	  			\node[label=left:\tiny{${H}$}] (u15) at (-4,2) {};
	  			
	  			\node[label=right:\tiny{${e_1}$}] (v1) at (-2.5,1) {};
	  			\node[label=below:\tiny{${e_2}$}] (v2) at (0,0) {};
	  			\node[label=below:\tiny{${e_3}$}] (v3) at (-3,-3) {};
	  			\node[label=above:\tiny{${e_4}$}] (v4) at (2,-2) {};
	  			\node[label=below:\tiny{${e_5}$}] (v5) at (-2,-2) {};
	  			\node[label=above:\tiny{${e_6}$}] (v6) at (-4,3) {};
	  			
	  			\draw (u1) -- (v1);
	  			\draw (u4) -- (v1);
	  			\draw (u5) -- (v1);
	  			\draw (u6) -- (v1);
	  			\draw (u12) -- (v1);
	  			\draw (u3) -- (v2);
	  			\draw (u5) -- (v2);
	  			\draw (u11) -- (v2);
	  			\draw (u7) -- (v2);
	  			\draw (u1) -- (v3);
	  			\draw (u13) -- (v3);
	  			\draw (u6) -- (v3);
	  			\draw (u8) -- (v3);
	  			\draw (u11) -- (v4);
	  			\draw (u3) -- (v4);
	  			\draw (u1) -- (v4);
	  			\draw (u9) -- (v4);
	  			\draw (u1) -- (v5);
	  			\draw (u4) -- (v5);
	  			\draw (u5) -- (v5);
	  			\draw (u14) -- (v5);
	  			\draw (u15) -- (v6);
	  			\draw (u12) -- (v6);
	  			
	  			
	  			
	  			
	  			\draw[->,] (-2,4) -- (0,4) [right] {};
	  			\draw[->,] (-1,-5) -- (-2,-5) [right] {};
	  			\draw[->,] (7,0) -- (7,-1) [right] {};
	  			\draw[->,] (-6,0) -- (-6,1) [right] {};
	  	\end{tikzpicture}}
	  	\hspace{.2cm}
	  	\subfloat[$\mathcal{I}(\tilde{\Gamma}_\mathcal{H}(\mathbb{Z}_2 \times \mathbb{Z}_2 \times \mathbb{Z}_3 \times \mathbb{Z}_3))$ on a projective plane]{\begin{tikzpicture}[scale=.65, every node/.style={circle, fill=black, minimum size=4pt, inner sep=0pt}]
	  			\draw[thick] (-6,-5) rectangle (7,4); 
	  			

	  			\node[label=above:\tiny{${\langle(1,0,0,0)\rangle}$}] (u1) at (-4,4) {};
	  			\node[label=below:\tiny{${\langle(1,0,0,0)\rangle}$}] (u2) at (3,-5) {};
	  			\node[label=above:\tiny{${\langle(0,1,0,0)\rangle}$}] (u3) at (-.5,4) {};
	  			\node[label=below:\tiny{${\langle(0,1,0,0)\rangle}$}] (u4) at (-.5,-5) {};
	  			\node[label=above:\tiny{${\langle(1,1,0,0)\rangle}$}] (u5) at (3,4) {};
	  			\node[label=below:\tiny{${\langle(1,1,0,0)\rangle}$}] (u6) at (-4,-5) {};
	  			\node[label=right:\tiny{${\langle(0,0,1,0)\rangle}$}] (u7) at (7,2) {};
	  			\node[label=left:\tiny{${\langle(0,0,1,0\rangle}$}] (u8) at (-6,-4) {};
	  			\node[label=right:\tiny{${\langle(0,0,0,1)\rangle}$}] (u9) at (7,0) {};
	  			\node[label=left:\tiny{${\langle(0,0,0,1\rangle}$}] (u10) at (-6,-2) {};
	  			\node[label=right:\tiny{${\langle(0,0,1,1)\rangle}$}] (u11) at (7,-2) {};
	  			\node[label=left:\tiny{${\langle(0,0,1,1)\rangle}$}] (u12) at (-6,0) {};
	  			\node[label=right:\tiny{${\langle(0,0,1,2)\rangle}$}] (u13) at (7,-4) {};
	  			\node[label=left:\tiny{${\langle(0,0,1,2)\rangle}$}] (u14) at (-6,2) {};
	  			\node[label=left:\tiny{${\langle(0,0,1,0),(0,0,0,1)\rangle}$}] (u15) at (1,0) {};
	  			\node[label=left:\tiny{${\langle(1,0,0,0),(0,1,0,0)\rangle}$}] (u16) at (2.5,1) {};
	  			\node[label=left:\tiny{${\langle(1,0,0,0),(0,0,1,0)\rangle}$}] (u17) at (-1.5,-3) {};
	  			
	  			\node[label=right:\tiny{${e_1}$}] (v1) at (-2.5,1) {};
	  			\node[label=right:\tiny{${e_2}$}] (v2) at (0,-3) {};
	  			\node[label=above:\tiny{${e_3}$}] (v3) at (4,0) {};
	  			\node[label=below:\tiny{${e_4}$}] (v4) at (2.4,-.3) {};
	  			\node[label=below:\tiny{${e_5}$}] (v5) at (-2,-2) {};

	  			\draw (u1) -- (v1);
	  			\draw (u3) -- (v1);
	  			\draw (u5) -- (v1);
	  			\draw (u8) -- (v1);
	  			\draw (u10) -- (v1);
	  			\draw (u12) -- (v1);
	  			\draw (u14) -- (v1);
	  			\draw (u2) -- (v2);
	  			\draw (u4) -- (v2);
	  			\draw (u6) -- (v2);
	  			\draw (u15) -- (v2);
	  			\draw (u7) -- (v3);
	  			\draw (u9) -- (v3);
	  			\draw (u11) -- (v3);
	  			\draw (u13) -- (v3);
	  			\draw (u16) -- (v3);
	  			\draw (u15) -- (v4);
	  			\draw (u16) -- (v4);
	  			\draw (u3) -- (v5);
	  			\draw (u5) -- (v5);
	  			\draw (u10) -- (v5);
	  			\draw (u12) -- (v5);
	  			\draw (u14) -- (v5);
	  			\draw (u17) -- (v5);
	  			
	  			
	  			
	  			\draw[->,] (-2,4) -- (-1,4) [right] {};
	  			\draw[->,] (-1,-5) -- (-2,-5) [right] {};
	  			\draw[->,] (7,2) -- (7,1) [right] {};
	  			\draw[->,] (-6,0) -- (-6,1) [right] {};
	  	\end{tikzpicture}}
	  	\hspace{.2cm}
	  	\subfloat[$\mathcal{I}(\tilde{\Gamma}_\mathcal{H}(\mathbb{Z}_{2} \times \mathbb{Z}_2 \times \mathbb{Z}_2))$ on a projective plane]{\begin{tikzpicture}[scale=.6, every node/.style={circle, fill=black, minimum size=4pt, inner sep=0pt}]
	  			\draw[thick] (-6,-5) rectangle (7,4); 
	  			

	  			\node[label=right:\tiny{${\langle(1,0,0)\rangle}$}] (u1) at (2,3) {};
	  			\node[label=below:\tiny{${\langle(0,1,0)\rangle}$}] (u2) at (-2,3) {};				
	  			\node[label=left:\tiny{${\langle(0,0,1)\rangle}$}] (u3) at (-3.5,3) {};
	  			\node[label=left:\tiny{${\langle(1,1,0)\rangle}$}] (u4) at (-6,0) {};
	  			\node[label=right:\tiny{${\langle(1,1,0)\rangle}$}] (u5) at (7,0) {};
	  			\node[label=left:\tiny{${\langle(1,0,1)\rangle}$}] (u6) at (2,-3) {};
	  			\node[label=left:\tiny{${\langle(0,1,1)\rangle}$}] (u7) at (-1,-3) {};
	  			\node[label=left:\tiny{${\langle(1,1,1)\rangle}$}] (u8) at (-3.5,-3) {};
	  			
	  			\node[label=below:\tiny{${e_1}$}] (v1) at (-2,-5) {};
	  			\node[label=above:\tiny{${e_1}$}] (v2) at (-2,4) {};
	  			\node[label=left:\tiny{${e_2}$}] (v3) at (-2,0) {};
	  			\node[label=left:\tiny{${e_3}$}] (v4) at (2,1) {};
	  			\node[label=left:\tiny{${e_4}$}] (v5) at (4,2) {};
	  			\node[label=right:\tiny{${e_5}$}] (v6) at (4,-1) {};
	  			\node[label=left:\tiny{${e_6}$}] (v7) at (-5,1) {};
	  			\node[label=right:\tiny{${e_7}$}] (v8) at (5,2) {};
	  			\node[label=right:\tiny{${e_8}$}] (v9) at (5,-2) {};
	  			
	  			\draw (u1) -- (v2);
	  			\draw (u2) -- (v2);
	  			\draw (u3) -- (v2);
	  			\draw (u5) -- (v2);
	  			\draw (u6) -- (v1);
	  			\draw (u7) -- (v1);
	  			\draw (u8) -- (v1);
	  			\draw (u3) -- (v3);
	  			\draw (u6) -- (v3);
	  			\draw (u7) -- (v3);
	  			\draw (u8) -- (v3);
	  			\draw (u3) -- (v4);
	  			\draw (u5) -- (v4);
	  			\draw (u6) -- (v4);
	  			\draw (u2) -- (v4);
	  			\draw (u1) -- (v5);
	  			\draw (u5) -- (v5);
	  			\draw (u6) -- (v5);
	  			\draw (u8) -- (v5);
	  			\draw (u1) -- (v6);
	  			\draw (u3) -- (v6);
	  			\draw (u5) -- (v6);
	  			\draw (u7) -- (v6);
	  			\draw (u1) -- (v7);
	  			\draw (u2) -- (v7);
	  			\draw (u3) -- (v7);
	  			\draw (u8) -- (v7);
	  			\draw (u1) -- (v8);
	  			\draw (u2) -- (v8);
	  			\draw (u6) -- (v8);
	  			\draw (u7) -- (v8);
	  			\draw (u2) -- (v9);
	  			\draw (u4) -- (v9);
	  			\draw (u7) -- (v9);
	  			\draw (u8) -- (v9);

	  			\draw[->,] (-2,4) -- (-1,4) [right] {};
	  			\draw[->,] (0,-5) -- (-1,-5) [right] {};
	  			\draw[->,] (7,0) -- (7,-1) [right] {};
	  			\draw[->,] (-6,0) -- (-6,1) [right] {};
	  	\end{tikzpicture}}
	  	\caption{}
	  	\label{proj2}
	  \end{figure}
\end{proof}
\begin{corollary}
	 $\tilde{\Gamma}_\mathcal{H}(G)$ is projective-planar if and only if $G \cong \mathbb{Z}_{p_1} \times \mathbb{Z}_{p_1}$, where $p_1$ is a prime or $ G \cong  Z_{2^\alpha} \times \mathbb{Z}_2$	or  $ G \cong \mathbb{Z}_2 \times \mathbb{Z}_2 \times \mathbb{Z}_{p_2^\beta}$, where $p_2$ is prime and  $\beta =1,2$, or $G \cong \mathbb{Z}_{3^2} \times \mathbb{Z}_3$.
\end{corollary}

  \noindent In nutshell, the comparison between $\mathscr{I}^c(G)^{\star}$ and its clique hypergraph $\tilde{\Gamma}_\mathcal{H}(G)$ is given in the following table.
  
   \vspace{.5cm}		
\resizebox{\textwidth}{4cm}{ $\begin{array}{|c|c|c|}
		\hline
		\textbf{Properties} & \textbf{Graph, $\mathscr{I}^c(G)^{\star}$, where $G$ is a finite non-cyclic abelian group } & \textbf{Clique hypergraph, $\tilde{\Gamma}_\mathcal{H}(G)$, where $G$ is a finite non-cyclic abelian group}\\ \hline  
		\textbf{Diameter} & \leq 3  & \leq 3 \\ \hline
		\textbf{Girth} & gr(\mathscr{I}^c(G)^{\star})= 3 \,\cite{devicomplement} & gr(\tilde{\Gamma}_\mathcal{H}(G) = 
		\begin{cases*}
			\infty,  \, \text{if} \,\, G \cong \mathbb{Z}_{p_1} \times \mathbb{Z}_{p_1},\\
			2, \,  \, \text{otherwise}.
		\end{cases*}\\ \hline
		\textbf{Star} & \text{Never a star}  \, \cite{devicomplement} & \tilde{\Gamma}_\mathcal{H}(\mathbb{Z}_{p_1} \times \mathbb{Z}_{p_1}) \\ \hline
		\textbf{Tree} & \text{Never a tree}\,\cite{devicomplement} & \tilde{\Gamma}_\mathcal{H}(\mathbb{Z}_{p_1} \times \mathbb{Z}_{p_1}) \\ \hline
		\textbf{Chromatic Number} & \geq 3 \, \cite{devicomplement}& 2 \,\, \text{for all $n$} \\ \hline
		\textbf{Planarity} & \mathscr{I}^c(\mathbb{Z}_{2^{\alpha}}\times \mathbb{Z}_2)^{\star}, \quad \mathscr{I}^c(\mathbb{Z}_3 \times \mathbb{Z}_3)^{\star}  \cite{devicomplement} & \tilde{\Gamma}_\mathcal{H}(\mathbb{Z}_{p_1} \times \mathbb{Z}_{p_1}), \quad  \tilde{\Gamma}_\mathcal{H}(\mathbb{Z}_{2^{\alpha}}\times \mathbb{Z}_2), \quad  \tilde{\Gamma}_\mathcal{H}(\mathbb{Z}_2 \times \mathbb{Z}_2 \times \mathbb{Z}_{p_2}^{\beta}), \, \text{$p_2$ is an odd prime}  \\ \hline
		\textbf{orientable genus 1} & - & \tilde{\Gamma}_\mathcal{H}(\mathbb{Z}_{3^2} \times \mathbb{Z}_3) \\ \hline
		\textbf{non-orientable genus 1} & - &  \tilde{\Gamma}_\mathcal{H}(\mathbb{Z}_{3^2} \times \mathbb{Z}_3)\\ \hline 
	\end{array}$}
\section*{Concluding Remarks}
Recently, the study of hypergraphs associated with group structures has gained significant attention. In this paper, we examined the clique hypergraph of the complement of intersection hypergraph on subgroups of a finite non-cyclic abelian group and analysed how it differs from the associated graph. We studied the basic structural properties of these hypergraphs  like girth, chromtic number, star structures, tree structures, etc. and characterized those hypergraphs  which can be embedded on various surfaces like plane, torus and projective plane. 
\section*{Acknowledgment}
This research was financially supported by the University Grants Commission (UGC), Government of India, through the award of the Junior Research Fellowship (JRF) under the National Eligibility Test (NET) (Ref. No. 241610038497) awarded to Ardra A N.
		
		\newpage

\end{document}